\documentclass[final]{siamltex}

\usepackage{amsmath,amsfonts,amssymb}
\usepackage[usenames]{color}

\usepackage[mathscr]{eucal}

\usepackage{stmaryrd}

\usepackage{amsbsy}

\usepackage{bm}

\usepackage{xspace}

\usepackage{paralist}

\usepackage{ifpdf}

\usepackage{graphicx}
\ifpdf
\usepackage{epstopdf}
\else
\usepackage{epsfig}
\fi

\usepackage{braket}

\usepackage{subfig}

\ifpdf
\usepackage[colorlinks,urlcolor=blue,citecolor=blue,linkcolor=blue]{hyperref}
\else
\newcommand{\href}[2]{{#2}}
\usepackage{url}
\fi

\usepackage{tabulary}

\usepackage{algorithm}
\usepackage{algpseudocode}

\ifpdf
\newcommand{\Sec}[1]{\hyperref[sec:#1]{\S\ref*{sec:#1}}} %
\newcommand{\App}[1]{\hyperref[sec:#1]{Appendix~\ref*{sec:#1}}} %
\newcommand{\Eqn}[1]{\hyperref[eq:#1]{{\rm (\ref*{eq:#1})}}} %
\newcommand{\Part}[1]{\hyperref[part:#1]{(\ref*{part:#1})}} %
\newcommand{\Fig}[1]{\hyperref[fig:#1]{Figure~\ref*{fig:#1}}} %
\newcommand{\Tab}[1]{\hyperref[tab:#1]{Table~\ref*{tab:#1}}} %
\newcommand{\Thm}[1]{\hyperref[thm:#1]{Theorem~\ref*{thm:#1}}} %
\newcommand{\Lem}[1]{\hyperref[lem:#1]{Lemma~\ref*{lem:#1}}} %
\newcommand{\Prop}[1]{\hyperref[prop:#1]{Property~\ref*{prop:#1}}} %
\newcommand{\Cor}[1]{\hyperref[cor:#1]{Corollary~\ref*{cor:#1}}} %
\newcommand{\Def}[1]{\hyperref[def:#1]{Definition~\ref*{def:#1}}} %
\newcommand{\Alg}[1]{\hyperref[alg:#1]{Algorithm~\ref*{alg:#1}}} %
\newcommand{\Ex}[1]{\hyperref[ex:#1]{Example~\ref*{ex:#1}}} %
\newcommand{\As}[1]{\hyperref[as:#1]{Assumption~{\rm\ref*{as:#1}}}} %
\newcommand{\Reg}[1]{\hyperref[as:#1]{Condition~\ref*{reg:#1}}} %
\newcommand{\AlgLine}[2]{\hyperref[alg:#1]{line~\ref*{line:#2} of Algorithm~\ref*{alg:#1}}}
\newcommand{\AlgLines}[3]{\hyperref[alg:#1]{lines~\ref*{line:#2}--\ref*{line:#3} of Algorithm~\ref*{alg:#1}}}
\else
\newcommand{\Sec}[1]{{\S\ref{sec:#1}}} %
\newcommand{\App}[1]{{Appendix~\ref{sec:#1}}} %
\newcommand{\Eqn}[1]{{(\ref{eq:#1})}} %
\newcommand{\Part}[1]{{(\ref{part:#1})}} %
\newcommand{\Fig}[1]{{Figure~\ref{fig:#1}}} %
\newcommand{\Tab}[1]{{Table~\ref{tab:#1}}} %
\newcommand{\Thm}[1]{{Theorem~\ref{thm:#1}}} %
\newcommand{\Lem}[1]{{Lemma~\ref{lem:#1}}} %
\newcommand{\Prop}[1]{{Property~\ref{prop:#1}}} %
\newcommand{\Cor}[1]{{Corollary~\ref{cor:#1}}} %
\newcommand{\Def}[1]{{Definition~\ref{def:#1}}} %
\newcommand{\Alg}[1]{{Algorithm~\ref{alg:#1}}} %
\newcommand{\Ex}[1]{{Example~\ref{ex:#1}}} %
\newcommand{\Reg}[1]{{R~\ref*{reg:#1}}} %
\fi

\newtheorem{assumption}[theorem]{Assumption}

\newenvironment{myproof}{\proof}{}
\newcommand{\myproofend}{\quad\endproof}

\newcommand{\Real}{\mathbb{R}}

\newcommand{\Tra}{^{\sf T}} %
\newcommand{\Inv}{^{-1}} %

\newcommand{\MI}[1]{\mathbf{#1}}
\newcommand{\MIn}[2]{\MI{#1}^{(#2)}}
\newcommand{\MInE}[3]{#1^{(#2)}_{#3}}

\newcommand{\Sn}[2]{#1^{(#2)}}
\newcommand{\Shatn}[2]{\hat #1^{(#2)}}

\newcommand{\V}[1]{{\bm{\mathbf{\MakeLowercase{#1}}}}} %
\newcommand{\Vbar}[1]{{\bm{\bar \mathbf{\MakeLowercase{#1}}}}} %
\newcommand{\Vhat}[1]{{\bm{\hat \mathbf{\MakeLowercase{#1}}}}} %
\newcommand{\VE}[2]{\MakeLowercase{#1}_{#2}} %
\newcommand{\Vn}[2]{\V{#1}^{(#2)}} %
\newcommand{\M}[1]{{\bm{\mathbf{\MakeUppercase{#1}}}}} %
\newcommand{\ME}[2]{\MakeLowercase{#1}_{#2}} %
\newcommand{\MC}[2]{\V{#1}_{#2}}
\newcommand{\Mhat}[1]{{\bm{\hat \mathbf{\MakeUppercase{#1}}}}} %
\newcommand{\MhatC}[2]{\Vhat{#1}_{#2}} %
\newcommand{\Mbar}[1]{{\bm{\bar \mathbf{\MakeUppercase{#1}}}}} %
\newcommand{\MbarC}[2]{\Vbar{#1}_{#2}} %
\newcommand{\Mn}[2]{\M{#1}^{(#2)}} %
\newcommand{\Mbarn}[2]{\Mbar{#1}^{(#2)}} %
\newcommand{\MnTra}[2]{\M{#1}^{(#2){{\sf T}}}} %
\newcommand{\MnE}[3]{\MakeLowercase{#1}^{(#2)}_{#3}} %
\newcommand{\MnC}[3]{\V{#1}^{(#2)}_{#3}} %
\newcommand{\MbarnC}[3]{\Vbar{#1}^{(#2)}_{#3}} %
\newcommand{\MnCTra}[3]{\V{#1}^{(#2){{\sf T}}}_{#3}} %
\newcommand{\T}[1]{\boldsymbol{\mathscr{\MakeUppercase{#1}}}} %
\newcommand{\Tbar}[1]{\boldsymbol{\bar \mathscr{\MakeUppercase{#1}}}} %
\newcommand{\That}[1]{\boldsymbol{\hat \mathscr{\MakeUppercase{#1}}}} %
\newcommand{\Ttilde}[1]{\boldsymbol{\tilde \mathscr{\MakeUppercase{#1}}}} %
\newcommand{\TE}[2]{\MakeLowercase{#1}_{\MI{#2}}} %

\newcommand{\Kron}{\otimes} %
\newcommand{\Khat}{\odot} %
\newcommand{\Hada}{\ast} %
\newcommand{\Divide}{\varoslash}

\newcommand{\Mz}[2]{\M{#1}_{(#2)}} %
\newcommand{\KT}[1]{\left\llbracket #1 \right\rrbracket} %
\newcommand{\KTsmall}[1]{\llbracket #1 \rrbracket} %
\newcommand{\KG}[1]{\langle #1 \rangle} %

\newcommand{\qtext}[1]{\quad\text{#1}\quad}

\newcommand{\FD}[2]{\frac{\partial #1}{\partial #2}}

\newcommand{\It}[1]{_{#1}}
\newcommand{\Itn}[2]{^{(#1)}\It{#2}}
\newcommand{\kl}{k_{\ell}}

\newenvironment{inlinemath}{$}{$}

\newcommand{\CH}[1]{\text{\rm conv}(#1)}

\newcommand{\LS}[2]{\mathcal{L}_{#1}(#2)}

\def\TX{\T{X}}
\def\MXn{\Mz{X}{n}}
\def\TM{\T{M}}
\def\MMn{\Mz{M}{n}}
\def\MAn{\Mn{A}{n}}
\def\MB{\M{B}}
\def\MBn{\Mn{B}{n}}
\def\MPi{\M{\Pi}}
\def\MPin{\Mn{\Pi}{n}}
\def\MPinTra{\MnTra{\Pi}{n}}
\def\MPhi{\M{\Phi}}
\def\MPhin{\Mn{\Phi}{n}}
\def\Vl{\V{\lambda}}
\def\ML{\M{\Lambda}}
\def\Vx{\V{x}}

\begin{document}
\title{On Tensors, Sparsity, and Nonnegative Factorizations%
  \thanks{The work of the first author was fully supported by the
    U.S.~Department of Energy Computational Science Graduate
    Fellowship under grant number DE-FG02-97ER25308. The work of the
    second author was funded by the applied mathematics program at the
    U.S.~Department of Energy and Sandia National Laboratories, a
    multiprogram laboratory operated by Sandia Corporation, a wholly
    owned subsidiary of Lockheed Martin Corporation, for the United
    States Department of Energy's National Nuclear Security
    Administration under contract DE-AC04-94AL85000.}}

\author{Eric C. Chi\footnotemark[2] \and Tamara G. Kolda\footnotemark[3]}
\maketitle

\renewcommand{\thefootnote}{\fnsymbol{footnote}}
\footnotetext[2]{Dept. Human Genetics, University of
  California, Los Angeles, CA. Email: eric.c.chi@gmail.com}
\footnotetext[3]{Sandia National Laboratories, Livermore, CA.
  Email: tgkolda@sandia.gov}
\renewcommand{\thefootnote}{\arabic{footnote}}

\begin{abstract}
  Tensors have found application in a variety of fields, ranging from
  chemometrics to signal processing and beyond. In this paper, we
  consider the problem of multilinear modeling of \emph{sparse count}
  data. Our goal is to develop a descriptive tensor factorization
  model of such data, along with appropriate algorithms and theory. To
  do so, we propose that the random variation is best described via a
  Poisson distribution, which better describes the zeros observed in
  the data as compared to the typical assumption of a Gaussian
  distribution. Under a Poisson assumption, we fit a model to observed
  data using the negative log-likelihood score. We present a new
  algorithm for Poisson tensor factorization 
  called CANDECOMP--PARAFAC Alternating Poisson Regression
  (CP-APR) that is based on a majorization-minimization approach. It
  can be shown that CP-APR is a generalization of the Lee-Seung
  multiplicative updates. We show how to prevent the algorithm from
  converging to non-KKT points and prove convergence of CP-APR under
  mild conditions. We also explain how to implement CP-APR for
  large-scale sparse tensors and present results on several data sets,
  both real and simulated.
\end{abstract}


\pagestyle{myheadings}
\thispagestyle{plain}
\markboth{\sc E.~C.~Chi and T.~G.~Kolda}{\sc Tensors, Sparsity, and
  Nonnegative Factorizations}

\section{Introduction}
\label{sec:introduction}

Tensors have found application in a variety of fields, ranging from
chemometrics to signal processing and beyond. In this paper, we
consider the problem of multilinear modeling of \emph{sparse count}
data. For instance, we may consider data that encodes the number of
papers published by each author at each conference per year for a
given time frame \cite{DuKoAc11}, the number  
of packets sent from one IP address to another using a specific port
\cite{SuTaFa06}, or to/from and term counts on emails
\cite{BaBeBr08}. Our goal is to develop a descriptive model of such
data, along with appropriate algorithms and theory.

Let $\T{X}$ represent an $N$-way data tensor of size $I_1 \times
I_2 \times \cdots \times I_N$. We are interested in an $R$-component
nonnegative CANDECOMP/PARAFAC \cite{CaCh70,Ha70} factor model 
\begin{equation}
  \label{eq:cp}
  \TM = \sum_{r=1}^R \lambda_r \;
  \MnC{A}{1}{r} \circ \cdots \circ \MnC{A}{N}{r},
\end{equation}
where $\circ$ represents outer product and $\MnC{A}{n}{r}$ represents the $r$th column of the nonnegative
\emph{factor matrix} $\MAn$ of size $I_n \times R$.  We refer to each
summand as a \emph{component}. Assuming each factor matrix has been
column-normalized to sum to one, we refer to the nonnegative
$\lambda_r$'s as \emph{weights}.

In many applications such as chemometrics
\cite{SmBrGe04}, we fit the model to the data using a least
squares criteria, implicitly assuming that the random variation in the
tensor data follows a Gaussian distribution. In the case of sparse count data,
however, the random variation is better described via a Poisson
distribution \cite{McNe89,Ro07}, i.e.,
\begin{displaymath}
  \TE{x}{i} \thicksim \text{Poisson}(\TE{m}{i})
\end{displaymath}
rather than $ \TE{x}{i} \thicksim N(\TE{m}{i}, \sigma_{\MI{i}}^2)$,
where the subscript $\MI{i}$ is shorthand for the multi-index $(i_1,
i_2, \dots, i_N)$.
In fact, a Poisson model is a much better explanation for the zero
observations that we encounter in sparse data --- these zeros just
correspond to events that were very unlikely to be observed.
Thus, we propose that rather than using the least squares (LS) error function given by
$\sum_{\MI{i}} | \TE{x}{i} - \TE{m}{i} |^2$, for count
data we should instead minimize the (generalized) Kullback-Leibler (KL) divergence
\begin{equation}
  \label{eq:nll}
  f(\TM) = \sum_{\MI{i}} \TE{m}{i} - \TE{x}{i} \log \TE{m}{i},
\end{equation}
which equals the negative log-likelihood of the observations up to an additive constant.
Unfortunately, minimizing KL~divergence is more difficult than LS error.

\subsection{Contributions}
Although other authors have considered fitting tensor data using KL divergence
\cite{WeWe01,CiZdChPl07,ZaPe11}, we offer the following contributions:
\begin{asparaitem}
\item %
  We develop alternating Poisson regression for 
  nonnegative CP model (CP-APR). The subproblems are solved
  using a majorization-minimization (MM) approach. If the algorithm is
  restricted to a single inner iteration per subproblem, it reduces to
  the standard Lee-Seung multiplicative for KL updates \cite{LeSe99,LeSe01}
  as extended to tensors by Welling and Weber
  \cite{WeWe01}. However, using multiple inner
  iterations is shown to accelerate the method, similar to what has
  been observed for LS \cite{GiGl11}
\item %
  It is known that the Lee-Seung multiplicative updates may converge
  to a non-stationary point \cite{GoZh05}, and Lin \cite{Li07a} has
  previously introduced a fix for the LS version of the Lee-Seung method.
  We introduce a different technique
  for avoiding \emph{inadmissible zeros} (i.e., zeros that violate
  stationarity conditions) that is only a
  trivial change to the basic algorithm and prevents convergence to
  non-stationary points. This technique is straightforward to adapt to the matrix and/or LS cases as well.
\item %
  Assuming the subproblems can be solved exactly, we prove convergence
  of the CP-APR algorithm. In particular, we can show convergence even
  for sparse input data and solutions on the boundary of the
  nonnegative orthant.
\item %
  We explain how to efficiently implement CP-APR for large-scale
  sparse data. Although it is well-known how to do large-scale sparse
  tensor calculations for the LS fitting function \cite{BaKo07},
  the Poisson likelihood fitting algorithm requires new sparse tensor kernels.
  To the best of our knowledge, ours is the
  first implementation of any KL-divergence-based method
  for large-scale sparse tensors.
\item %
  We present experimental results showing the effectiveness of the
  method on both real and simulated data. In fact, the Poisson
  assumption leads quite naturally to a generative model for sparse
  data.
\end{asparaitem}

\subsection{Related Work}
\label{sec:related}

Much of the past work in nonnegative matrix and tensor analysis has
focused on the LS error \cite{PaTa94, Pa97, BrDe97, GoZh05,
  KiPa08, KiSrDh08, FrHa08}, which corresponds to an assumption of
normal independently identically distributed (i.i.d.) noise.  The
focus of this paper is KL divergence, which
corresponds to maximum likelihood estimation under an independent Poisson
assumption; see \Sec{Poisson}.  The seminal work in this
domain are the papers of Lee and Seung \cite{LeSe99, LeSe01}, which
propose very simple \emph{multiplicative} update formulas for both
LS and KL divergence, resulting
in a very low cost-per-iteration.  Welling and Weber \cite{WeWe01}
were the first to generalize the Lee and Seung algorithms to
nonnegative tensor factorization (NTF).  Applications of NTF based on
KL-divergence include EEG analysis \cite{MoHaPaAr06} and sound source
separation \cite{FiCrCo05}.
We note that generalizations of KL divergence have also been
proposed in the literature, including Bregman divergence \cite{DhSr06,dSLi08,LiCo09}
and beta divergence~\cite{CiZdChPl07, FeId11}.

In terms of convergence, Lin \cite{Li07a} and Gillis and Glienur
\cite{GiGl08} have shown convergence of two different modified
versions of the Lee-Seung method for LS. 
Finesso and Spreij \cite{FiSp06} (tensor extension in \cite{ZaPe11}) 
have shown convergence of the
Lee-Seung method for KL divergence; however, we show later that
numerical issues arise if the iterates come near to the boundary.
This is related to the problems demonstrated by Gonzalez and Zhang \cite{GoZh05}
that show, in the case of LS loss, the Lee and
Seung method can converge to non-KKT points; 
we show a similar example for KL divergence in 
\Sec{misconvergence}.

Our convergence theory is not focused on the Lee-Seung algorithm but
rather on a Gauss-Seidel approach. The closest work is that of Lin
\cite{Li07} in which he considers the matrix problem in the least
squares sense; in the same paper, he dismisses the KL divergence
problem as ill-defined but we address that issue in this paper by
showing that the convex hull of the level sets of the KL divergence problem are compact.

\section{Notation and Preliminaries}
\label{sec:notation}

\subsection{Notation}

Throughout, scalars are denoted by lowercase letters ($a$), vectors
by boldface lowercase letters ($\V{v}$), matrices by boldface capital letters ($\M{A}$), and higher-order
tensors by boldface Euler script letters ($\T{X}$).
We let $\V{e}$ denotes the vector of all ones and $\M{E}$ denotes the matrix of all ones.
The $j$th column of a matrix $\M{A}$ is denoted by $\MC{A}{j}$.
We use multi-index notation so that a boldface $\MI{i}$ represents the
index $(i_1,\dots,i_N)$.
We use subscripts to denote iteration index for infinite
sequences, and the difference between its use for an entry and its use
as an iteration index should be clear by context. 

The notation $\|\cdot\|$ refers to the two-norm for vectors or Frobenious norm for matrices, i.e., the sum of the squares of the entries. The notation $\|\cdot\|_1$ refers to the one-norm, i.e., the sum of the absolute values of the entries.

The outer product is denoted by $\circ$.
The symbols $\Hada$ and $\Divide$ represents elementwise
multiplication and division, respectively. 
The symbol $\Khat$ denotes Khatri-Rao matrix
multiplication.
The mode-$n$ matricization or unfolding of a tensor
$\T{X}$ is denoted by $\Mz{X}{n}$. %
See \App{notation-details} for further details on these operations.

\subsection{The Poisson Distribution and KL Divergence}
\label{sec:Poisson}
In statistics, count data is often best described as following a
Poisson distribution.
For a general discussion of the Poisson distribution,
see, e.g., \cite{Ro07}. We summarize key facts here.

A random variable $X$ is said to have a Poisson distribution with
parameter $\mu > 0$ if it takes integer values $x=0,1,2,\dots$ with
probability
\begin{equation}
  \label{eq:Poisson}
  P(X=x) = \frac{e^{-\mu} \mu^{x}}{x!}.
\end{equation}
The mean and variance of $X$ are both $\mu$; therefore, the variance
increases along with the mean, which seems like a reasonable
assumption for count data.  It is also useful to note that the sum of
independent Poisson random variables is also Poisson. This is
important in our case since each Poisson parameter is a multilinear
combination of the model parameters. We contrast Poisson and Gaussian
distributions in \Fig{gaussian_and_poisson}. Observe that there is good agreement between the distributions for larger values of the mean,
$\mu$. For small values of $\mu$, however, the match is not as strong and
the Gaussian random variable can take on negative values.

\begin{figure}
  \centering
  \includegraphics[width=.7\textwidth,trim=0 0 0 0,clip]{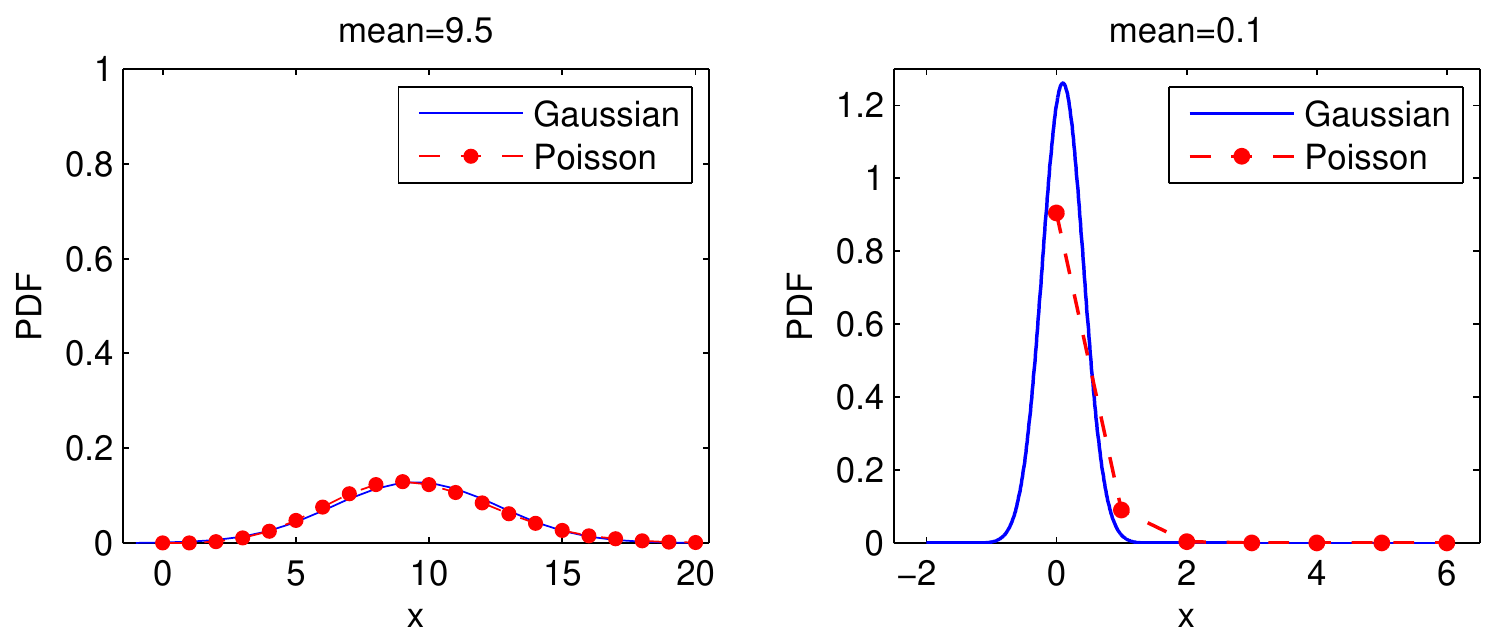}
  \caption{Illustration of Gaussian and Poisson distributions for two parameters. For both examples, we assume that the variance of the Gaussian is equal to the mean $m$.}
  \label{fig:gaussian_and_poisson}
\end{figure}

We can determine the optimal Poisson parameters by maximizing the
likelihood of the observed data. Let $\V{x}$ be a vector of observations
and let $\V{\mu}$ be the vector of Poisson parameters. (We assume that
$\mu_i$'s are not independent, else the function would entirely decouple
in the parameters to be estimated.)
Then the negative of the log of the likelihood function for
\Eqn{Poisson} is the KL divergence
\begin{equation}
  \label{eq:log-likelihood}
  \sum_i \mu_i - x_i \log \mu_i,
\end{equation}
excepting the addition of the  constant term $\sum_i \log(x_i!)$, 
which is omitted.  

Because we are working with sparse data, there are many instances for
which we expect $x_i = 0$, which leads to some ambiguity in
\Eqn{log-likelihood} if $\mu_i =0$. We assume throughout that
$0 \cdot \log(\mu) = 0$ {for all} $\mu \geq 0$.
This is for notational convenience; else, we would write
\Eqn{log-likelihood} as
\begin{inlinemath}
  \sum_i \mu_i - \sum_{i:x_i \neq 0} x_i \log \mu_i.
\end{inlinemath}

\section{CP-APR: Alternating Poisson Regression}
\label{sec:fitting}

In this section we introduce the CP-APR algorithm for fitting a
nonnegative \emph{Poisson tensor decomposition (PTF)} to count data. The algorithm employs an
alternating optimization scheme that sequentially optimizes
one factor matrix while holding the others fixed; this is 
nonlinear Gauss-Seidel applied to the PTF problem.  The subproblems are solved via a
majorization-minimization (MM) algorithm, as described in \Sec{subproblem}.

\subsection{The Optimization Problem}

Our optimization problem is defined as
\begin{gather}\label{eq:nlp}
  \min \; f(\TM) \equiv \sum_{\MI{i}} \TE{m}{i} - \TE{x}{i} \log \TE{m}{i}
  \quad \text{s.t. }
  \TM =\KT{\Vl; \Mn{A}{1},\dots,\Mn{A}{N}} \in \Omega,\\
  \label{eq:Omega}
  \begin{gathered}
  \qtext{where} \Omega = \Omega_{\lambda} \times \Omega_1 \times \cdots \times
  \Omega_n \qtext{with}  \\
  \Omega_{\lambda} = [0,+\infty)^R
  \qtext{and}
  \Omega_n = \Set{ \M{A} \in [0,1]^{I_n \times R} |
    \| \MC{a}{r} \|_1 = 1 \text{ for } r=1,\dots,R }.
  \end{gathered}
\end{gather}
Here $\TM =\llbracket{\Vl; \Mn{A}{1},\dots,\Mn{A}{N}}\rrbracket$ is
shorthand notation for \Eqn{cp} \cite{BaKo07}.  Depending on context,
$\TM$ represents the tensor itself or its constituent parts. For
example, when we say $\TM \in \Omega$, it means that that the factor
matrices have stochasticity constraints on the columns.

The function $f$ is not finite on all of $\Omega$. For example, if
there exists $\MI{i}$ such that $\TE{m}{i} = 0$ and $\TE{x}{i} > 0$,
then $f(\TM) = +\infty$.
If $\TE{m}{i} > 0$ for all $\MI{i}$ such that $\TE{x}{i}>0$, however,
then we are guaranteed that $f(\TM)$ is finite.
Consequently, we will generally wish to restrict ourselves to a domain
for which  $f(\TM)$ is finite. We define
\begin{equation}
  \label{eq:OmegaZeta}
  \Omega(\zeta) \equiv \CH{\Set{ \TM \in \Omega | f(\TM) \leq \zeta }},
\end{equation}
where $\CH{\cdot}$ denotes the convex hull.
We observe that $\Omega(\zeta) \subset \Omega$ (strict subset) since,
for example, the all-zero model is not in $\Omega(\zeta)$.
The following lemma states that $\Omega(\zeta)$ is compact for
any $\zeta > 0$; the proof is given in \App{proof-OmegaZeta}.

\begin{lemma}\label{lem:OmegaZeta}
  Let $f$ be as defined in \Eqn{nlp} and $\Omega(\zeta)$ be as defined in \Eqn{OmegaZeta}.
  For any $\zeta > 0$, $\Omega(\zeta)$ is compact.
\end{lemma}

\subsection{CP-APR Main Loop: Nonlinear Gauss-Seidel}

We solve problem \Eqn{nlp} via an alternating approach, holding all
factor matrices constant except one.
Consider the problem for the $n$th factor matrix.
We note that there is scaling
ambiguity that allows us to express the same $\TM$ in different ways,
i.e.,
\begin{gather}\label{eq:M2}
  \TM = \KT{\Mn{A}{1},\dots,\Mn{A}{n-1}, \Mn{B}{n},
    \Mn{A}{n+1},\dots,\Mn{A}{N}}\\
  \text{where} \quad
  \label{eq:Bn}
  \MBn = \MAn \ML \qtext{ and } \ML = \diag(\Vl).
\end{gather}
The weights in \Eqn{M2} are omitted 
because they are absorbed into the $n$th mode. 
From \cite{BaKo07}, we can express $\TM$ as $\MMn = \MBn  \MPin$
where $\MBn$ is defined in \Eqn{Bn} and
\begin{equation}
  \label{eq:Pi}
  \MPin \equiv \left( \Mn{A}{N} \Khat \cdots  \Khat \Mn{A}{n+1} \Khat
    \Mn{A}{n-1} \Khat \cdots \Khat \Mn{A}{1} \right)\Tra.
\end{equation}
Thus, we can rewrite the objective function in \Eqn{nlp} as
\begin{displaymath}
  f(\TM) = \V{e}\Tra \left[ \MBn\MPin - \MXn \Hada \log
    \left(\MBn\MPin\right) \right] \V{e},
\end{displaymath}
where $\V{e}$ is the vector of all ones, $\Hada$ denotes the
elementwise product, and the $\log$ function is applied elementwise.
We note that it is convenient to update $\MAn$ and $\Vl$
simultaneously since the resulting constraint on $\MBn$ is simply
$\MBn \geq 0$.

Thus, at each inner iteration of the Gauss-Seidel algorithm, we
optimize $f(\TM)$ restricted to the $n$th block, i.e.,
\begin{equation}
  \label{eq:subproblem}
  \MBn = \arg \min_{\MB \geq 0} f_n(\MB) \equiv
  \V{e}\Tra \left[ \MB \MPin - \MXn \Hada \log
    \left(\MB \MPin\right) \right] \V{e}.
\end{equation}
The updates for $\Vl$ and $\MAn$ come directly from $\MBn$. Note that
some care must be taken if an entire column of $\MBn$ is zero; if the
$r$th column is zero, then we can set $\lambda_r = 0$ and
$\MnC{B}{n}{r}$ to an arbitrary nonnegative vector that sums to one.
The full procedure is given in \Alg{outer}; this is a variant (because
of the handling of $\Vl$) of nonlinear Gauss-Seidel. We note that the
scaling and unscaling of the factor matrices is common in alternating algorithms,
though not always explicit in the algorithm statement. 
There are many variations of this basic device; for instance, in the
context of the LS version of NTF,
\cite[Algorithm 2]{FrHa08} collects the scaling information into an explicit
scaling vector that is ``amended'' after each inner iteration 

\begin{algorithm}[t]
  \caption{CP-APR Algorithm (Ideal Version)}
  \label{alg:outer}
  Let $\T{X}$ be a tensor of size $I_1 \times \dots \times I_N$.
  Let $\TM = \KTsmall{\Vl;\Mn{A}{1},\cdots,\Mn{A}{N}}$ be an
  initial guess for an $R$-component model such that $\TM \in
  \Omega(\zeta)$ for some $\zeta > 0$.
  \begin{algorithmic}[1]
    \Repeat
    \For{$n = 1, \ldots, N$}
    \State $\M{\Pi} \gets  \left( \Mn{A}{N} \Khat \cdots \Khat
      \Mn{A}{n+1} \Khat \Mn{A}{n-1} \Khat \cdots \Khat \Mn{A}{1} \right)\Tra$
    \State $\MB \gets  \arg \displaystyle\min_{\MB \geq 0}
    \V{e}\Tra \left[ \MB \MPi - \MXn \Hada \log \left( \MB \MPi \right)
    \right] \V{e}$
    \Comment subproblem
    \State $\Vl \leftarrow \V{e}\Tra \MB$
    \State $\MAn \leftarrow \MB \ML\Inv$			
    \EndFor
    \Until{convergence}
  \end{algorithmic}
\end{algorithm}

We defer the proof of convergence until \Sec{convergence}, but we
discuss how to check for convergence here.
First, we mention an assumption that is important to the theory and also arguably practical.
Let
\begin{equation}\label{eq:S}
  \mathcal{S}_i^{(n)} = \Set{j | (\MXn)_{ij} >0}
\end{equation}
denote the set of indices of columns for which the $i$th row of
$\MXn$ is non-zero.
If $N=3$, then $\Mz{x}{1}(i,:)$ corresponds to a vectorization of
the $i$th horizontal slice of $\TX$, $\Mz{x}{2}(i,:)$ to a
vectorization of the $i$th lateral slice, and $\Mz{x}{3}(i,:)$ to a
vectorization of the $i$th frontal slice.
More generally, we can think of vectorizing ``hyperslices'' with
respect to each mode.

\begin{assumption}
  \label{as:full_row_rank}
  The rows of the submatrix $\MPin(:,\mathcal{S}_i^{(n)})$
  (i.e., only the columns corresponding to nonzero rows in
  $\Mz{X}{n}$ are considered) are linearly independent
  for all $i=1,\dots,I_n$ and $n = 1,\dots,N$.
\end{assumption}

\As{full_row_rank} implies that $\lvert
\mathcal{S}_i^{(n)} \rvert \geq R$ for all $i$. Thus, we need to
observe at least $R \cdot \max_{n} I_n$ counts in the data
tensor $\T{X}$, and the counts need to be sufficiently distributed
across $\T{X}$. Consequently, the conditions appeal to our intuition that
there are concrete limits on how sparse the data tensor can be with
respect to how many parameters we wish to fit.
If, for example, we had $\Mz{X}{1}(i,:) = 0$, it
is clear that we can remove element $i$ from the first dimension entirely
since it contributes nothing. We are making a stronger requirement:
each element in each dimension must have at least $R$ nonzeros in its corresponding
hyperslice.

A potential problem is that \As{full_row_rank} depends on the current
iterate, which we cannot predict in advance. However, we observe that
if $\Vl > 0$ and the
factor matrices have random uniform [0,1] positive entries and $R
\leq \min_{n} \prod_{m\neq n} I_m$, then this condition is satisfied
with probability one\footnote{We can actually appeal to a weaker
assumption; if the entries are drawn from any distribution that is 
absolutely continuous with respect to the Lebesgue measure on [0,1] then
the condition is satisfied with probability one.}. This condition can
be checked as the iterates progress.

The matrix \begin{equation}
  \label{eq:Phi}
  \MPhin \equiv \left[\MXn \Divide \left(\MBn\MPin\right)\right] \MPinTra,
\end{equation}
with $\Divide$ denoting elementwise division, will come up repeatedly
in the remainder of the paper. For instance, we observe that
the partial derivative of $f$ with respect to $\MAn$ is
$\partial{f}/\partial{\MAn}  = (\M{E} - \MPhin) \ML$,
where $\M{E}$ is the matrix of all ones.
Consequently, the matrix $\MPhin$ plays a role in checking convergence
as follows.

\begin{theorem}
  \label{thm:kkt}
  If $\Vl > 0$ and $\TM = \KTsmall{\Vl;\Mn{A}{1},\dots,\Mn{A}{N}} \in
  \Omega(\zeta)$ for some $\zeta > 0$, then $\TM$
  is a Karush-Kuhn-Tucker (KKT) point of \Eqn{nlp} if and only if
  \begin{equation}\label{eq:kkt-conditions}
    \min \left( \MAn,  \M{E} - \MPhin \right) = 0
    \text{ for } n = 1,\dots,N.
  \end{equation}
\end{theorem}
\begin{proof}
  Since $\Vl > 0$, we can assume that $\Vl$ has been absorbed into
  $\Mn{A}{m}$ for some $m$. Thus, we can replace the constraints $\Vl \in
  \Omega_{\lambda}$ and $\Mn{A}{m} \in \Omega_n$ with $\Mn{B}{m} \geq
  0$.   In this case, the partial derivatives are
  \begin{equation}
    \FD{f}{\Mn{B}{m}}  = \M{E} - \Mn{\Phi}{m} \qtext{and}
    \FD{f}{\MAn}  = \left(\M{E} - \MPhin\right) \ML  \text{ for } n
    \neq m.
  \end{equation}
  Since $\TM \in \Omega(\zeta)$ for some $\zeta > 0$, we know that not
  all elements of $\TM$ are zero; thus, the set of active constraints
  are linearly independent.   The following conditions define a KKT
  point \cite{NoWr99}:
  \begin{equation}\label{eq:kkt-full-conditions}
    \begin{gathered}
       \M{E} - \Mn{\Phi}{m} - \Mn{\Upsilon}{m} = 0, \\
       (\M{E} - \MPhin) \ML - \Mn{\Upsilon}{n} -
       \V{e}(\Vn{\eta}{n})\Tra = 0, \;
       \V{e}\Tra\MAn = 1
       \qtext{for} n \neq m \\
       \MAn \geq 0, \;
       \Mn{\Upsilon}{n} \geq 0, \;
       \Mn{\Upsilon}{n} \Hada \MAn = 0
       \qtext{for} n \neq m \\
       \Mn{B}{m} \geq 0, \;
       \Mn{\Upsilon}{m} \geq 0, \;
       \Mn{\Upsilon}{m} \Hada \Mn{B}{m} = 0.
    \end{gathered}
  \end{equation}
  Here $\Mn{\Upsilon}{n}$ are the Lagrange multipliers for the
  nonnegativity constraints and $\Vn{\eta}{n}$ are the Lagrange
  multipliers for the stochasticity constraints.

  If $\TM= \KG{\Vl;\Mn{A}{1},\dots,\Mn{A}{N}}$ is a KKT point, then
  from \Eqn{kkt-full-conditions}, we have that $\Mn{\Upsilon}{m} =  \M{E} -
  \Mn{\Phi}{m} \geq 0$, $\Mn{B}{m} \geq 0$, and $\Mn{\Upsilon}{m} \Hada \Mn{B}{m} = 0$. Thus,
  $\min ( \Mn{A}{m} \ML,  \M{E} - \Mn{\Phi}{m} ) = 0$. Since
  $\Vl > 0$ and $m$ is arbitrary, \Eqn{kkt-conditions} follows immediately.

  If, on the other hand, \Eqn{kkt-conditions} is satisfied, choosing
  $\Mn{\Upsilon}{m} = \M{E} - \Mn{\Phi}{m}$, and
  $\Mn{\Upsilon}{n} = ( \M{E} - \Mn{\Phi}{n}) \ML$ and
  $\Vn{\eta}{n} = 0$ for $n \neq m$ satisfies the KKT conditions in
  \Eqn{kkt-full-conditions}. Hence, $\TM$ must be a KKT point.
\end{proof}

Observe that the condition $\Vl >0$ makes $\Vl$ moot in the KKT
conditions --- this reflects the scaling ambiguity that is inherent in
the model.

From \Thm{kkt} and because feasibility is always maintained, we can check for convergence by verifying
$| \min ( \MAn , \M{E} - \MPhin ) | \leq \tau$
for $n = 1,\dots,N$,
where $\tau > 0$ is some specified convergence tolerance.

\subsection{Convergence Theory for CP-APR}
\label{sec:convergence}

We require the strict convexity of $f$ in each of the block coordinates. This is ensured under \As{full_row_rank}.

\begin{lemma}[Strict convexity of subproblem]
  \label{lem:strict-convexity}
  Let $f_n(\cdot)$ be the function $f$ restricted to the $n$th block as defined in \Eqn{subproblem}.
  If \As{full_row_rank} is satisfied, then $f_n(\MB)$
  is strictly convex over $\mathcal{B}_n = \{\MB \in [0, +\infty)^{I_n \times R} : \MB\Mn{\Pi}{n} \neq \M{0}\}$.
\end{lemma}

\begin{proof}
  In the proof, we drop the $n$'s for convenience.
  First note that $\mathcal{B}$ is convex.
  Let $\M{C} = \MB\Tra$. We can rewrite \Eqn{subproblem}
  as %
  $\min f(\M{C}\Tra) =
  \sum_{ij}  \MC{C}{i}\Tra\MC{\pi}{j} - x_{ij}
  \log ( \MC{C}{i}\Tra\MC{\pi}{j} )$ subject to $\M{C} \geq 0$.
  Hence,
  it is sufficient to show that the function
  $\hat f(\M{C}) = - \sum_{ij} x_{ij} \log (\V{c}_{i}\Tra\V{\pi}_{j})$
  is strictly convex over the convex set $\mathcal{C} = \{\M{C} \in
  [0, +\infty)^{R \times I_n} : \M{C}\Tra\MPi \neq \M{0}\}$.
  Fix $\Mbar{C}, \Mhat{C} \in \mathcal{C}$ such
  that $\Mbar{C} \neq \Mhat{C}$.
  Since the inner product is affine and $\log$ is a strictly concave
  function, we need only show that there exists some $i$ and $j$ such
  that $x_{ij} \neq 0$ and $\MhatC{c}{i}\Tra \MC{\pi}{j} \neq
  \MbarC{C}{i}\Tra \MC{\pi}{j}$.  We know at least one column must
  differ since $\Mbar{C} \neq \Mhat{C}$; let $i$ correspond to that
  column and define $\V{d} = \MhatC{c}{i} - \MbarC{C}{i} \neq 0$. By
  \As{full_row_rank}, we know that $\MPi(:,S_i)$ has full row
  rank. Thus, there exists a column $j$ of $\MPi$ such that $x_{ij}
  \neq 0$ and $\V{d}\Tra\MC{\pi}{j} \neq 0$.  Hence, the claim.
\end{proof}

Here we state our main convergence result. Although this result
assumes that the subproblems can be solved exactly (which is not the
case in practice), it gives some idea as to the convergence behavior
of the method. We follow the reasoning of the proof of convergence of
nonlinear Gauss-Seidel \cite[Proposition 3.9]{BeTs89}, adapted here
for the way that $\Vl$ is handled. 

\begin{theorem}[Convergence of CP-APR]
  \label{thm:cp-apr-convergence}
  Suppose that $f(\TM)$ is strictly convex with respect to each block
  component and that it is minimized exactly for each block component
  subproblem of CP-APR. Let $\TM\It{*}$ be a limit point of the sequence 
  $\{\TM\It{k}\}$ such that $\Vl\It{*} > 0$. 
  Then $\TM\It{*}$ is a KKT
  point of \Eqn{nlp}.
\end{theorem}

\begin{proof}
  Let $\TM\It{k} = \KG{\Vl\It{k}, \Mn{A}{1}\It{k}, \dots,
    \Mn{A}{N}\It{k}}$ be the $k$th iterate produced by the
  \emph{outer} iterations of \Alg{outer}.
  Define $\T{Z}\Itn{n}{k}$ to be
  the $n$th iterate in the inner loop of outer iteration $k$ with the
  $\Vl$-vector absorbed into the $n$th factor, i.e.,
  \begin{displaymath}
    \T{Z}\Itn{n}{k} = \KG{\Mn{A}{1}\It{k+1}, \dots, \Mn{A}{n-1}\It{k+1},
      \Mn{B}{n}\It{k+1}, \Mn{A}{n+1}\It{k}, \dots, \Mn{A}{N}\It{k}},
  \end{displaymath}
  where $\MBn\It{k+1}$ is the solution to the $n$th subproblem at
  iteration $k$. This defines $\MAn\It{k+1}$ to be the column-normalized
  version of $\MBn\It{k+1}$, i.e.,
  $\MAn\It{k+1} = \MBn\It{k+1} (\diag( \MBn\It{k+1} \V{e}))\Inv$.
  Taking advantage of the scaling ambiguity to shift the weights between factors yields
  \begin{align*}
    f(\T{Z}\Itn{n}{k}) 
    &= f(\KG{\Mn{A}{1}\It{k+1}, \dots, \Mn{A}{n-1}\It{k+1},
      \Mn{A}{n}\It{k+1} \diag( \MBn\It{k+1} \V{e}), \Mn{A}{n+1}\It{k}, \dots, \Mn{A}{N}\It{k}}), \\
    &= f(\KG{\Mn{A}{1}\It{k+1}, \dots, \Mn{A}{n-1}\It{k+1},
      \Mn{A}{n}\It{k+1} , \Mn{A}{n+1}\It{k} \diag( \MBn\It{k+1} \V{e}), \dots, \Mn{A}{N}\It{k}}), \\
    & \geq f(\KG{\Mn{A}{1}\It{k+1}, \dots, \Mn{A}{n-1}\It{k+1},
      \Mn{A}{n}\It{k+1} , \Mn{B}{n+1}\It{k+1}, \dots, \Mn{A}{N}\It{k}}) = f(\T{Z}\Itn{n+1}{k}).
  \end{align*}
  Observe that
  \begin{inlinemath}
    \T{Z}\Itn{N}{k} = \KG{\Mn{A}{1}\It{k+1}, \dots, \Mn{A}{N-1}\It{k+1},
      \Mn{A}{N}\It{k+1} \diag(\Vl\It{k+1})},
  \end{inlinemath}
  so there is a correspondence between $\T{Z}\Itn{N}{k}$ and
  $\TM\It{k+1}$ such that $f(\T{Z}\Itn{N}{k}) = f(\TM\It{k+1})$.
  For convenience, we define
  \begin{inlinemath}
    \T{Z}\Itn{0}{k}  = \KG{\Mn{A}{1}\It{k}\diag(\Vl\It{k}),  \Mn{A}{2}\It{k},
      \dots, \Mn{A}{N}\It{k}}.
  \end{inlinemath}
  Since we assume the subproblem is solved exactly at each iteration,
  we have
  \begin{equation}\label{eq:BT3.14}
    f(\TM\It{k}) \geq
    f(\T{Z}\Itn{1}{k}) \geq
    f(\T{Z}\Itn{2}{k}) \geq
    \cdots
    f(\T{Z}\Itn{N-1}{k}) \geq
    f(\TM\It{k+1})
    \text{ for all } k.
  \end{equation}

  Recall that $\Omega(\zeta)$ is compact by \Lem{OmegaZeta}.  Since
  the sequence $\{\TM\It{k}\}$ is contained in the set
  $\Omega(\zeta)$, it must have a convergent subsequence. We let
  $\{\kl\}$ denote the indices of that convergent subsequence and
  $\TM\It{*} = \KG{\Vl\It{*}, \Mn{A}{1}\It{*}, \dots,
    \Mn{A}{N}\It{*}}$ denote its limit point. By continuity of $f$,
  \begin{inlinemath}
    f(\TM\It{\kl}) \rightarrow f(\TM\It{*}).
  \end{inlinemath}

  We first show that $\| \Mn{A}{1}\It{\kl + 1} - \Mn{A}{1}\It{\kl} \|
  \rightarrow 0$. Assume the contrary, i.e., that it does not converge to
  zero. Let $\gamma\It{\kl} = \| \T{Z}\Itn{1}{\kl} - \T{Z}\Itn{0}{\kl}
  \|$. By possibly restricting to a subsequence of $\{\kl\}$, we may
  assume there exists some $\gamma\It{0} > 0$ such that $\gamma(\kl) \geq
  \gamma\It{0}$ for all $\ell$.
  Let $\Mn{S}{1}\It{\kl} = (\T{Z}\Itn{1}{\kl} - \T{Z}\Itn{0}{\kl}) /
  \gamma\It{\kl}$; then $\T{Z}\Itn{1}{\kl} =  \T{Z}\Itn{0}{\kl} +
  \gamma\It{\kl} \Mn{S}{1}\It{\kl}$, $\|\Mn{S}{1}\It{\kl}\| = 1$, and
  $\Mn{S}{1}\It{\kl}$ differs from zero only along the first block component.
  Notice that $\{\Mn{S}{1}\It{\kl}\}$ belong to a compact set and
  therefore has a limit point $\Mn{S}{1}\It{*}$. By restricting to a
  further subsequence of $\{\kl\}$, we assume that $\Mn{S}{1}\It{\kl}
  \rightarrow \Mn{S}{1}\It{*}$

  Let us fix some $\epsilon \in [0,1]$. Notice that $0 \leq \epsilon
  \gamma\It{0} \leq \gamma\It{\kl}$. Therefore, $\T{Z}\Itn{0}{\kl} +
  \epsilon\gamma\It{0} \Mn{S}{1}\It{\kl}$ lies on the line segment joining
  $\T{Z}\Itn{0}{\kl}$ and $\T{Z}\Itn{0}{\kl} +
  \gamma\It{\kl} \Mn{S}{1}\It{\kl} = \T{Z}\Itn{1}{\kl}$ and belongs to
  $\Omega(\zeta)$ because $\Omega(\zeta)$ is
  convex. Using the convexity of $f$ w.r.t.\@ the first
  block component and the fact that $\T{Z}\Itn{1}{\kl}$ minimizes
  $f$ over all $\T{Z}$ that differ from $\T{Z}\Itn{1}{\kl}$ in the
  first block component, we obtain
  \begin{displaymath}
  f(\T{Z}\Itn{1}{\kl})
  = f(\T{Z}\Itn{0}{\kl} + \gamma\It{\kl} \Mn{S}{1}\It{\kl})
  \leq f(\T{Z}\Itn{0}{\kl} + \epsilon\gamma\It{0} \Mn{S}{1}\It{\kl})
  \leq f(\T{Z}\Itn{0}{\kl}).
  \end{displaymath}
  Since $f(\T{Z}\Itn{0}{\kl}) = f(\TM\It{\kl}) \rightarrow f(\TM\It{*})$, equation
  \Eqn{BT3.14} shows that $f(\T{Z}\Itn{1}{\kl})$ also converges to
  $f(\TM\It{*})$. Taking limits as $\ell$ tends to infinity, we obtain
  \begin{displaymath}
    f(\TM\It{*})
    \leq f(\T{Z}\Itn{0}{*} + \epsilon\gamma\It{0} \Mn{S}{1}\It{*})
    \leq f(\TM\It{*}),
  \end{displaymath}
  where $\T{Z}\Itn{0}{*}$ is just $\T{M}\It{*}$ with $\Vl\It{*}$ absorbed
  into the first component.
  We conclude that
  \begin{inlinemath}
    f(\TM\It{*})
    =  f(\T{Z}\Itn{0}{*} + \epsilon\gamma\It{0} \Mn{S}{1}\It{*})
  \end{inlinemath}
  for every $\epsilon \in [0,1]$. Since $\gamma\It{0}\Mn{S}{1}\It{*}
  \neq 0$, this contradicts the strict convexity of $f$ as a function
  of the first block component. This contradiction establishes that
  $\| \Mn{A}{1}\It{\kl + 1} - \Mn{A}{1}\It{\kl} \|
  \rightarrow 0$. In particular, $\T{Z}\Itn{1}{\kl}$ converges to
  $\T{Z}\Itn{0}{*}$.

  By definition of $\T{Z}\Itn{1}{\kl}$ and the assumption that each
  subproblem is solved exactly, we have
  \begin{displaymath}
    f(\T{Z}\Itn{1}{\kl})
    \leq
    f(\KG{ \MB, \Mn{A}{2}\It{\kl}, \dots, \Mn{A}{N}\It{\kl}})
    \text{ for all } \MB \geq 0.
  \end{displaymath}
  Taking limits as $\ell \rightarrow \infty$, we obtain
  \begin{displaymath}
    f(\T{M}\It{*}) \leq
    f(\KG{ \MB, \Mn{A}{2}\It{*}, \dots, \Mn{A}{N}\It{*}})
    \text{ for all } \MB \geq 0.
  \end{displaymath}
  In other words, $\Mn{B}{1}\It{*} = \Mn{A}{1}\It{*} \diag( \Vl\It{*}
  )$ is the minimizer of $f$ with respect to the first block
  components with the remaining components are fixed at
  $\Mn{A}{2}\It{*}$ through $\Mn{A}{N}\It{*}$.
  From the KKT conditions \cite{NoWr99},  we have that
  \begin{displaymath}
    \Mn{B}{1}\It{*} \geq 0,
    \quad
    \FD{f}{\Mn{B}{1}}(\Mn{B}{1}\It{*}) \geq 0,
    \quad
    \Mn{B}{1}\It{*} \Hada \FD{f}{\Mn{B}{1}}(\Mn{B}{1}\It{*}) = 0.
  \end{displaymath}
  In turn, since $\Vl\It{*} > 0$, we have
  $\min ( \Mn{A}{1}\It{*}, \M{E} - \Mn{\Phi}{1}\It{*} ) = 0$.

  Repeating the previous argument shows that
  $\| \Mn{A}{2}\It{\kl + 1} - \Mn{A}{2}\It{\kl} \|
  \rightarrow 0$ and that
  $\min ( \Mn{A}{2}\It{*}, \M{E} - \Mn{\Phi}{2}\It{*} ) = 0$.
  Continuing inductively, 
  ${\min ( \Mn{A}{n}\It{*}, {\M{E}-\Mn{\Phi}{n}\It{*}} ) = 0}$
  for $n = 1,\dots,N$.
  Thus, by \Thm{kkt},  $\T{M}\It{*}$ is a KKT point of $f(\TM)$.
\end{proof}

Before proceeding to the discussion solving the subproblem, we point out
that remarkably very little is assumed about the objective function $f$ in \Thm{cp-apr-convergence}. The proof
required that $f$ is differentiable, strictly convex in each of its
block components, and there is a $\xi > 0$ such that the level set $\Omega(\xi)$ is compact.
The upshot is that \Thm{cp-apr-convergence} applies equally well to
other choices of $f$ corresponding to other members in the family of  
beta distributions that are convex, namely the divergences that
correspond to $\beta \in [1, 2]$~\cite{FeId11}. 
In fact, 
it was also observed in \cite{FrHa08} that  ``rescaling does not interfere with the
convergence of the Gauss--Seidel iterations'' (in the context of 
the LS formulation of NTF).

\section{Solving the CP-APR Subproblem via Majorization-Minimization}
\label{sec:subproblem}

The basic idea of a majorization-minimization (MM) algorithm is to
convert a hard optimization problem (e.g., non-convex and/or
non-differentiable) into a series of simpler ones (e.g., smooth
convex) that are easy to minimize and that majorize the original
function, as follows.

\begin{definition}
  Let $f$ and $g$ be real-valued functions on $\mathbb{R}^n$ and $\mathbb{R}^n\times\mathbb{R}^n$, respectively.  We say
  that $g$ {\em majorizes} $f$ at $\V{x}\in \Real^n$ if $g(\V{y}, \V{x})
  \geq f(\V{y})$ for all $\V{y} \in \Real^n$ and $g(\V{x}, \V{x}) =
  f(\V{x})$.
\end{definition}

If $f(\Vx)$ is the function to be optimized and $g(\cdot,\Vx)$ majorizes
$f$ at $\Vx$, the basic MM iteration is
\begin{inlinemath}
  \Vx\It{k+1} = \arg \min_{\V{y}} g(\V{y}, \Vx\It{k}).
\end{inlinemath}
It is easy to
see that such iterates always take non-increasing steps with respect to
$f$ since
\begin{inlinemath}
  f(\Vx\It{k+1}) \leq g(\Vx\It{k+1}, \Vx\It{k}) \leq g(\Vx\It{k}, \Vx\It{k}) = f(\Vx\It{k}),
\end{inlinemath}
where $\Vx\It{k}$ is the current iterate and $\Vx\It{k+1}$ is the optimum computed
at that iterate.

Consider the $n$th subproblem in \Eqn{subproblem}. Here we drop the
$n$'s for convenience %
so that \Eqn{subproblem} reduces to
\begin{equation}\label{eq:subproblem-simple}
  \min_{\MB \geq 0} f(\MB) \equiv 
  \V{e}\Tra \left [\MB\M{\Pi} - \M{X} \Hada \log(\MB \M{\Pi}) \right ]\V{e}.
\end{equation}
Recall that $\M{X}$ is the nonnegative data tensor reshaped to a
matrix of size $I \times J$, $\MPi$ is a nonnegative matrix of size $R
\times J$ with rows that sum to 1, and $\MB$ is a nonnegative matrix
of size $I \times R$. For clarity in the ensuing discussion, we also restate \As{full_row_rank} in terms of the local
variables for this section as follows:
\begin{assumption}
  \label{as:full_row_rank_mm}
  The rows of the submatrix $\MPi\left(:,\Set{j | \M{X}_{ij} > 0}\right)$
  (i.e., only the columns corresponding to nonzero rows in
  $\M{X}$ are considered) are linearly independent
  for all $i=1,\dots,I$.
\end{assumption}

According to \As{full_row_rank_mm}, for every $i$ there is at
least one $j$ such that $x_{ij} > 0$. Thus, we can assume that we have $\Mbar{B}
\geq 0$ such that $f(\Mbar{B})$ is finite.  We now introduce the majorization used in our subproblem solver.
This majorization is also a special case of the one
derived in \cite{FeId11} when $\beta = 1$ and
has a long history in image reconstruction that predates its 
use in NMF \cite{Lucy1974, ShVa82, LaCa84}.
The objective $f$ is majorized at $\Mbar{B}$ by the function
\begin{equation}
   \label{eq:majorization}
  g(\M{B}, \Mbar{B}) = \sum_{rij}
  \left [ {b}_{ir} {\pi}_{rj} -
  \alpha_{rij} x_{ij} \log\left( \frac{b_{ir} \pi_{rj}}{\alpha_{rij}}
    \right) \right]
  \qtext{where}
  \alpha_{rij} = \frac{\bar b_{ir}\pi_{rj}}{\sum_r \bar b_{ir} \pi_{rj}}.  
\end{equation}
The proof of this fact is straightforward and thus relegated to
\App{proof-mm}.  The advantage of this majorization is that the problem is now
completely separable in terms of $b_{ir}$, i.e., the individual
entries of $\M{B}$. Moreover, $g(\cdot, \Mbar{B})$ has a unique
global minimum with an analytic expression, given by $\MB \Hada \MPhi$,
where $\MPhi$ is as defined in \Eqn{Phi} and depends on $\MB$.
A proof is provided in \App{proof-mm-unique-global-min}. The MM algorithm iterations are then defined by
\begin{equation}
  \label{eq:mm-iterate}
  \MB\It{k+1} = \psi(\MB\It{k}) \equiv \MB\It{k} \Hada \MPhi(\MB\It{k}),
  \qtext{where}
  \MPhi(\MB\It{k}) = 
  \left[ \M{X} \Divide \left( \MB\It{k} \MPi \right) \right] \MPi
\end{equation}
and $\M{X}$ and $\MPi$ come from \Eqn{subproblem-simple}.
If $\MB\It{0} \geq 0$, clearly $\MB\It{k} \geq 0$ for all $k$. 
Observe that
$\nabla f(\MB) = \M{E} - \MPhi(\M{B})$. We discuss in \Sec{stopping} how to exploit this simple relationship to quickly compute
stopping rules for the algorithm. The MM algorithm to solve the Gauss-Seidel subproblem of Line 4 in \Alg{outer}
is given in \Alg{CPAPR}.

\begin{algorithm}[t]
  \caption{Subproblem Solver for \Alg{outer}}
  \label{alg:CPAPR}
  \begin{algorithmic}[1]
    \State $\MB \gets \MAn\ML$ \label{line:B}
    \Repeat \Comment{subproblem loop} \label{line:a}
    \State $\MPhi \leftarrow \left( \Mz{X}{n} \Divide (\MB \MPi )
    \right ) \MPi\Tra$ \label{line:Phi}
    \State $\MB \leftarrow \MB \Hada \MPhi$ \label{line:B-update}
    \Until{convergence} \label{line:b}
  \end{algorithmic}
\end{algorithm}

The monotonic decrease in objective function does not guarantee that the MM iterates will
converge to the desired global minimizer of the
subproblem. Nonetheless, the following theorem shows that, under mild
conditions on the starting point $\MB\It{0}$ (discussed further in \Sec{zeros}), the MM iterates
will converge to the unique global minimum of
\Eqn{subproblem-simple}. 
The proof follows the reasoning of the convergence proof of an algorithm for fitting a regularized Poisson regression problem given in \cite{La90} and is given in \App{subproblem-convergence}.

\begin{theorem}[Convergence of MM algorithm]
\label{thm:subproblem-convergence}
  Let $f$ be as defined in \Eqn{subproblem-simple} and assume
  \As{full_row_rank_mm} holds, let $\MB\It{0}$
  be a nonnegative matrix such that $f(\MB\It{0})$ is finite
  and $\left( \MB\It{0} \right)_{ir} > 0$ for all $(i,r)$ such that
  $(\MPhi(\MB\It{*}))_{ir} > 1$, and let the sequence $\{\MB\It{k}\}$
  be defined as in \Eqn{mm-iterate}.
  Then  $\{\MB\It{k}\}$ converges to the global minimizer of $f$.
\end{theorem}

Note that we make a modest but very useful generalization 
of existing results by allowing iterates to be on (or very close to) the boundary.
Prior convergence
results, including \cite{La90, FiSp06, ZaPe11} assume 
that all iterates are strictly positive. Though true in exact
arithmetic, in numerical computations it is not uncommon for some
iterates to become zero numerically. In \Sec{zeros}, we show how to
ensure the condition on $\MB\It{0}$ holds in practice.
\section{CP-APR Implementation Details}
\label{sec:implementation}
The previous algorithms omit many details and
numerical checks that are needed in any practical implementation.
Thus, \Alg{CPAPR-detailed} provides a detailed version that can be
directly implemented. A highlight  of this implementation is the
``inadmissible zero'' avoidance, which fixes a the problem of getting
stuck at a zero value with multiplicative updates.

\begin{algorithm}[t]
  \caption{Detailed CP-APR Algorithm}
  \label{alg:CPAPR-detailed}
  Let $\T{X}$ be a tensor of size $I_1 \times \dots \times I_N$.
  Let $\TM = \KG{\Vl;\Mn{A}{1},\cdots,\Mn{A}{N}}$ be an
  initial guess for an $R$-component model such that $\TM \in
  \Omega(\zeta)$ for some $\zeta > 0$. \\
  Choose the following parameters:
  \begin{itemize}
  \item $k_{\max}$ = Maximum number of outer iterations
  \item $\ell_{\max}$ = Maximum number of inner iterations (per outer iteration)
  \item $\tau$ = Convergence tolerance on KKT conditions (e.g., $10^{-4}$)
  \item $\kappa$ = Inadmissible zero avoidance adjustment (e.g., $0.01$)
  \item $\kappa_{\rm tol}$ = Tolerance for identifying a potential
    inadmissible zero (e.g., $10^{-10}$)
  \item $\epsilon$ = Minimum divisor to prevent divide-by-zero (e.g.,  $10^{-10}$)
  \end{itemize}
  \begin{algorithmic}[1]
    \For{$k = 1,2,\dots,k_{\max}$}
    \State \texttt{isConverged} $\gets$ true
    \For{$n = 1, \ldots, N$}
    \State \label{line:scooch-a}
    $\M{S}(i,r) \gets
    \begin{cases}
      \kappa, & \text{if } k > 1,
      \MAn(i,r) < \kappa_{\rm tol}, \text{ and } \MPhin(i,r) > 1,  \\
      0, & \text{otherwise}
    \end{cases}
    $
    \State \label{line:scooch-b} 
    $\MB \gets (\MAn  + \M{S}) \ML$
    \State $\M{\Pi} \gets  \left( \Mn{A}{N} \Khat \cdots \Khat
      \Mn{A}{n+1} \Khat \Mn{A}{n-1} \Khat \cdots \Khat \Mn{A}{1} \right)\Tra$ \label{line:update_pi}
    \For{$\ell=1,2,\dots,\ell_{\max}$} \Comment{subproblem loop} 
    \State $\MPhin \gets \left( \Mz{X}{n} \Divide (\max(\MB \MPi,\epsilon) )
    \right ) \MPi\Tra$ \label{line:multiplicative_update}
    \If{ $| \min ( \MB , \M{E} - \MPhin ) | < \tau$ } \label{line:tau}
    \State break
    \EndIf
    \State \texttt{isConverged} $\gets$ false
    \State $\MB \gets \MB \Hada \MPhin$
    \EndFor
    \State $\Vl \gets \V{e}\Tra \MB$
    \State $\MAn \gets \MB \ML\Inv$			
    \EndFor
    \If{ \texttt{isConverged} = true }
    \State break
    \EndIf
    \EndFor
  \end{algorithmic}
\end{algorithm}

\subsection{Lee-Seung is a special case of CP-APR}
If we only take one iteration of the subproblem loop (i.e., setting
$\ell_{\max} = 1$), then CP-APR is the Lee-Seung multiplicative
update algorithm for the KL divergence.  Thus, we can view the
Lee-Seung algorithm as a special case of our algorithm where we do not
solve the subproblems exactly; quite the contrary, we only take one
step towards the subproblem solution.

\subsection{Inadmissible Zero Avoidance}
\label{sec:zeros}
A well-known problem with multiplicative updates is that some
elements may get ``stuck'' at zero; see, e.g., \cite{GoZh05}. 
For example, if $\MnE{A}{n}{ir} =
0$, then the multiplicative updates %
will
never change it. In many cases, a zero entry may be the correct
answer, so we want to allow it. In other cases, though, the zero entry
may be incorrect in the sense that it does not satisfy the KKT
conditions, i.e., $\MnE{A}{n}{ir} = 0$ but
$1 - \MnE{\Phi}{n}{ir} < 0$.
We refer to these values as \emph{inadmissible zeros}.
We correct this problem before we enter into the multiplicative
update phase of the algorithm. 
In \AlgLines{CPAPR-detailed}{scooch-a}{scooch-b},
any inadmissible zeros (or
near-zeros) are ``scooched'' away from zero and into the interior. The amount of the
scooch is controlled by the user-defined parameter $\kappa$.  
The condition in \Thm{subproblem-convergence} is exactly that the starting
point should not have any zeros that are ultimately inadmissible. 
If we discover that a sequence of iterates leads to an inadmissible
zero (or almost-zero), we restart the method by restarting the method
with a new starting point.
This adjustment prevents convergence to non-KKT points.
Note that all the quantities needed to perform the check are
precomputed and that there is no change to the algorithm besides
adjusting a few zero entries in the current factor matrix.
The fix for the inadmissible
zeros is compatible with the Lee-Seung algorithm  for LS error  as
well. 

Lin \cite{Li07a} has made a similar observation in the LS
case and applied changes to his gradient descent version of the
Lee-Seung method. Our correction is different and is directly
incorporated into the multiplicative update scheme rather than
requiring a different update formula.
Gillis and Glineur \cite{GiGl08} proposed a
more drastic fix by restricting 
the factor matrices to have entries in $[\epsilon, \infty)$ for
some small positive $\epsilon$. 
Avoiding all zeros clearly rules out the possibility of getting stuck
at an inadmissible zero, but does so at the expense of 
eliminating any hope of obtaining sparse factor matrices, a desirable
property  in many applications.

\subsection{Practical Considerations on Convergence}
\label{sec:stopping}

The convergence conditions on the subproblem require that 
$\min ( \MBn , \M{E} - \MPhin ) = 0$.
We do not
require the value to be exactly zero but instead check that it is
smaller in magnitude than the user-defined parameter $\tau$.
We break out of the subproblem loop as soon as this condition is satisfied.

From \Thm{kkt}, we can check for overall convergence by verifying
\Eqn{kkt-conditions}.
We do not want to calculate this at the end of every $n$-loop because
it is expensive. Instead, we know that the iterates will stop changing
once we have converged and so we can validate the convergence of all
factor matrices by checking that no factor matrix has been modified
and every subproblem has converged. 

\subsection{Sparse Tensor Implementation}
\label{sec:sparse}
Consider a large-scale sparse tensor that is too large to be stored as a dense tensor requiring $\prod_n I_n$
memory. In this case, we can store the tensor as a sparse tensor as
described in \cite{BaKo07}, requiring only $(N+1) \cdot
\mathop{\rm nnz}(\TX)$ memory.

The elementwise division in the update of $\MPhi$ 
requires that we divide the tensor (in matricized
form) $\M{X}$ by the current model estimate (in matricized form)
$\M{M} = \MB\MPi$. Unfortunately, we cannot afford to store $\M{M}$
explicitly as a dense tensor because it is the same size as $\TX$.  In
fact, we generally cannot even form $\MPi$ explicitly because it
requires almost as much storage as $\TM$.
We observe, however, that we need only calculate the values of $\M{M}$
that correspond to nonzeros in $\M{X}$. 

Let $P = \text{nnz}(\T{X})$. Then we can store the sparse tensor
$\T{X}$ as a set of values and multi-indices, $(\Sn{v}{p}, \MIn{i}{p})$ for
$p=1,\dots,P$. In order to avoid forming the current model estimate,
$\T{M}$, as a dense object, we will store only selected rows of $\MPi$,  one
per nonzero in $\T{X}$; we denote these rows by
$\Vn{w}{p}$ for $p=1,\dots,P$.
The $p$th vector is given by the elementwise product of rows of the
factor matrices, i.e.,
\begin{displaymath}
  \Vn{w}{p} = 
  \Mn{A}{1}(\MInE{i}{p}{1},:) \Hada \cdots \Hada
  \Mn{A}{n-1}(\MInE{i}{p}{n-1},:) \Hada
  \Mn{A}{n+1}(\MInE{i}{p}{n+1},:) \Hada \cdots \Hada
  \Mn{A}{N}(\MInE{i}{p}{N},:) .
\end{displaymath}
In order to determine $\That{X} = \T{X} \Divide \T{M}$ in the calculation
of $\MPhi$, we proceed as follows. The tensor $\That{X}$ will have the
same nonzero pattern as $\T{X}$, and we let $\Shatn{v}{p}$ denote its
values. It can be determined that
\begin{displaymath}
  \Shatn{v}{p} = \Sn{x}{p} / \left< \Vn{w}{p} ,
    \Mn{A}{n}(\MInE{i}{p}{n},:) \right>.
\end{displaymath}
To calculate $\MPhi = \Mhat{X} \MPi$, we simply have
\begin{displaymath}
  \MPhi(i',r) = \sum_{p : \MInE{i}{p}{n} = i'} \Shatn{v}{p} \Vn{w}{p}(r).
\end{displaymath}
The storage of the $\Vn{w}{p}$ for $p=1,\dots,P$ vectors and the
entries $\Shatn{v}{p}$ requires $(R+1)P$ additional storage.

\section{Numerical Results for CP-APR}
\label{sec:numerical}

\subsection{Comparison of Objective Functions  for Sparse Count Data}
\label{sec:compare_objectives}
We contend that, for sparse count data, KL divergence \Eqn{nll} is a better
objective function. To support our claim, we
consider simulated data where we know the correct answer. 
Specifically, we consider a 3-way tensor ($N=3$) of size $1000 \times 800 \times
600$ and $R=10$ factors. It will be generated from a model $\T{M} =
\KTsmall{\Vl; \Mn{A}{1},\dots,\Mn{A}{N}}$. The entries of the vector
$\Vl$ are selected uniformly at random from $[0,1]$.  Each factor
matrix $\Mn{A}{n}$ is generated as follows:
\begin{inparaenum}[(1)]
\item For each column in $\Mn{A}{n}$, randomly select 10\% (i.e.,
  $1/R$) of the entries uniformly at random from the
  interval $[0,100]$.
\item The remaining entries are selected uniformly at random from
   $[0,1]$.
\item Each column is scaled so that its 1-norm is 1 (i.e., its sum is 1).
\end{inparaenum}
An ``observed" tensor can be thought of as the outcome of tossing
$\nu \ll \prod I_n$ balls into $\prod I_n$ empty urns where  
each entry of the tensor corresponds to an urn. For each ball, we first draw a factor $r$
with probability $\lambda_r/\sum\lambda_r$. 
The indices $(i,j,k)$ are selected randomly proportional to $\MnC{A}{n}{r}$ for $n=1,2,3$.
In other words, 
the ball is then tossed into the $(i,j,k)$th urn 
with probability $\MnE{A}{1}{ir}\MnE{A}{2}{jr}\MnE{A}{3}{kr}$. In this
manner, the balls are allocated across the urns independently of each other.
This procedure generates entries $\TE{x}{i}$ that are each distributed
as $\text{Poisson}(\TE{m}{i})$.
We adjust the final $\V{\lambda}$ so that the scale matches that of
$\T{X}$, i.e., $\Vl \gets \nu \Vl / \|\Vl\|$. 
We generate problems where the number of observations ranges from 480,000 (0.1\%) down to 24,000
(0.005\%). 
Recall that \As{full_row_rank} implies that the absolute
minimum number of observations is $R \cdot \max_n I_n = 10,000$.
We have used very few observations, as real problems do indeed tend
to be this sparse.

\begin{table}[t]
  \centering
  \begin{tabular}{r@{\;}c|*2{c@{\;\;}c}|*2{c@{\;\;}c}}
    \multicolumn{2}{c}{}
    & \multicolumn{4}{c}{Least Squares}  
    & \multicolumn{4}{c}{KL Divergence}     \\
    \multicolumn{2}{c}{}
    & \multicolumn{2}{c}{Lee-Seung LS}  
    & \multicolumn{2}{c}{CP-ALS} 
    & \multicolumn{2}{c}{Lee-Seung KL} 
    & \multicolumn{2}{c}{CP-APR}\\
    \multicolumn{2}{c}{Observations} 
& FMS & \#Cols & FMS & \multicolumn{1}{@{\;}c}{\#Cols} & FMS & \#Cols & FMS & \#Cols \\ \hline
480000 & (0.100\%) & 0.58 & 6.4  & 0.71 & 7.3  & 0.89 & 8.7  & 0.96 & 9.5 \\ 
240000 & (0.050\%) & 0.51 & 5.4  & 0.72 & 7.4  & 0.83 & 8.2  & 0.91 & 9.2 \\ 
48000 & (0.010\%) & 0.37 & 3.8  & 0.59 & 6.3  & 0.76 & 7.5  & 0.80 & 7.9 \\ 
24000 & (0.005\%) & 0.33 & 3.5  & 0.51 & 5.7  & 0.72 & 6.6  & 0.74 & 6.9 \\ \hline
  \end{tabular}
  \caption{Accuracy comparison (mean of 10 trials) using the factor
    match score (FMS) and the number of columns correctly identified
    in the first factor matrix.}
  \label{tab:simulated}
\end{table}

\Tab{simulated} shows comparisons of four methods.
The first two are optimizing LS:
Lee-Seung for LS and alternating LS with no
nonnegativity constraints (CP-ALS).
The last two are optimizing KL divergence:
Lee-Seung for KL divergence and our method (CP-APR).
We have also tested the modified Lee-Seung method of Finesso and Spreij
\cite{FiSp06,ZaPe11}, but it is only a scaled version of the Lee-Seung
method for KL divergence and gave nearly identical results which are omitted.
All implementations are from Version 2.5 of Tensor Toolbox for MATLAB
\cite{TTB_Software,BaKo07,BaBeBr08}; exact parameter settings are 
provided in \App{compare_objectives-app}.
We report the factor match score (FMS), a measure in $[0,1]$ of how close the
computed solution is to the true solution. A value of 1 is ideal. Since the FMS measure is
somewhat abstract, we also report the number of columns in the first
factor matrix such that the cosine of the angle between the true
solution and the computed solution is greater than 0.95.  A value of
10 is ideal since we have used $R=10$.
The reported values are averages over 10 problems. 
See \App{compare_objectives-app} for precise formulas for both measures.
Although these problems are extremely sparse, all methods are able to
correctly identify components in the data. Overall, the methods
optimizing KL divergence are superior to those optimizing least
squares.
We also observe that CP-APR is an improvement compared to
Lee-Seung~KL; we provide later evidence that this improvement is more
likely due to the inadmissible zero fix than the extra inner
iterations (which provide a benefit of enhanced speed rather than accuracy).

\subsection{Fixing Misconvergence of Lee-Seung}
\label{sec:misconvergence}

We demonstrate the effectiveness of our simple fix for avoiding
inadmissible zeros, as described in \Sec{zeros}.  
Our technique is based on the same observation on inadmissible zeros
as in Lin \cite{Li07a}, but the change to the algorithm is
different. 
As in \cite{GoZh05}, we consider fitting a rank-10 bilinear model for
a $25 \times 15$ dense positive matrix with entries drawn
independently and uniformly 
from $[0,1]$. 
We apply CP-APR  using $\ell_{\max} = 1, \tau = 10^{-15}, 
\epsilon = 0, \kappa_{\text{tol}} = 100 \cdot
\epsilon_{\text{mach}}$. 
We do two runs: one with $\kappa = 0$, corresponding to the standard
Lee-Seung (KL version) algorithm, and the other with 
$\kappa = 10^{-10}$ to move away from inadmissible zeros. In
both runs we use the same strictly positive initial guess.
\Fig{misconvergence} shows the magnitude of the KKT residual over more than $10^{5}$ iterations. 
When $\kappa > 0$, the sequence clearly convergences. On the other
hand when $\kappa = 0$ the iterates appear to get stuck at a non-KKT
point.  Closer inspection of the factor matrix iterates reveals a
single offending inadmissible zero, i.e., its partial
derivative is $-0.0016$ but should be nonnegative. Hence, we use
positive values of $\kappa$ in our experiments.

\begin{figure}
\centering
\includegraphics[scale=0.6]{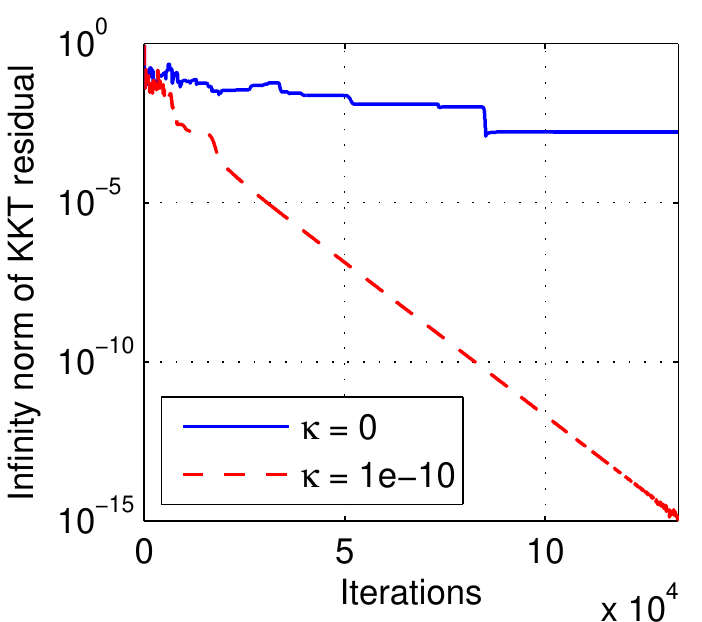}
\caption{Lee-Seung permitting inadmissible zeros (blue solid line) and avoiding inadmissible zeros (red dashed line).}
\label{fig:misconvergence}
\end{figure}

\subsection{The Benefit of Extra Inner Iterations}

We show that increasing the maximum number of inner iterations
$\ell_{\max}$ can accelerate the convergence in \Tab{timing_plus}. Recall that $\ell_{\max}
= 1$ corresponds to the Lee-Seung algorithm \cite{LeSe99,WeWe01}.  We consider a 3-way
tensor ($N = 3$) of size $500 \times 400 \times 300$ and $R = 5$
factors. We generate 100 problem instances from 100 randomly generated
models $\T{M} = \KTsmall{\Vl; \Mn{A}{1},\dots,\Mn{A}{N}}$ as described
in \Sec{compare_objectives} with 0.1\% observations.
We compare CP-APR with $\ell_{\max} = 1, 5,$ and $10$. 
and the other parameters set as  $k_{\max} = 10^6$, $\tau
= 10^{-4}$, $\kappa=10^{-8}$,
$\kappa_{\text{tol}}=100 \cdot \epsilon_{\rm mach}$,
$\epsilon=0$.
We track both the
number of multiplicative updates
(\AlgLine{CPAPR-detailed}{multiplicative_update})
and the CPU time using the MATLAB command
\texttt{cputime}. The experiments were performed on an iMac computer
with a 3.4 GHz Intel Core i7 processor and 8 GB of RAM. 
\Tab{FMS_timing} reports the FMS scores as we vary $\ell_{\max}$, and
we observe that the value of $\ell_{\max}$ does not significantly impact
accuracy. 
However, we observe that increasing $\ell_{\max}$ can decrease the
overall work and runtime.
Tables~\ref{tab:counts} and \ref{tab:timing}
present the average number of
multiplicative updates and total run times respectively. The
distribution of updates and times was highly skewed as some problems
required a substantial number of iterations. Nonetheless, we 
see a monotonic decrease in the number of updates and time as
$\ell_{\max}$ increases. The differences are more substantial when
comparing wall clock time. The reason for the
disproportionate decrease in wall-clock time compared to the tally of
updates is that the cost of the calculation of $\M{\Pi}$ (in
\AlgLine{CPAPR-detailed}{update_pi}) is amortized over 
all the subproblem iterations.

\newcommand{\CC}[1]{\multicolumn{1}{c}{#1}}
\begin{table}[t]
\centering
\subfloat[Factor match scores]{\label{tab:FMS_timing}
\begin{tabular}{cccc}
 {\bf $\ell_{\max}$} & 1 & 5  & 10 \\ \hline
Median & 0.9858 & 0.9858 &0.9862 \\
Mean   & 0.9483 & 0.9514 & 0.9603 \\ \hline
\end{tabular}
}
\\
\subfloat[Number of multiplicative updates]{
  \label{tab:counts}
\begin{tabular}{crrr}
 \CC{$\ell_{\max}$} & \CC{1} & \CC{5}  & \CC{10} \\ 
  \hline
  Median & 9819 & 7655 & 7290 \\ 
  Mean & 16370 & 11710 & 11660 \\ 
   \hline
\end{tabular}
}
~~
\subfloat[Time (seconds)]{
  \label{tab:timing}
\begin{tabular}{crrr}
 \CC{$\ell_{\max}$} & \CC{1} & \CC{5}  & \CC{10} \\ 
  \hline
  Median & 168.70 & 68.98 & 55.00 \\ 
  Mean & 299.60 & 106.10 & 87.92 \\ 
   \hline
\end{tabular}
}
  \caption{CP-APR with different values of $\ell_{\max}$ for sparse count
    data over 100 trials.}
  \label{tab:timing_plus}
\end{table}

\subsection{Enron Data}

We consider the application of CP-APR to email data from the infamous
Federal Energy Regulatory Commission (FERC) investigation of Enron
Corporation. We use the version of the dataset prepared by Zhou et
al.\@ \cite{ZhGoMaWa07} and further processed by Perry and Wolfe
\cite{PeWo10}, which includes detailed profiles on the employees. The
data is arranged as a three-way tensor $\T{X}$ arranged as sender $\times$ receiver
$\times$ month, where entry $(i,j,k)$ indicates the number of
messages from employee $i$ to employee $j$ in month $k$.  The original
data set had 38,388 messages (technically, there were only 21,635
messages but some messages were sent to multiple recipients and so are
counted multiple times) exchanged between 156 employees over 44 months
(November 1998 -- June 2002). We preprocessed the data, removing
months that had fewer than 300 messages and removing any employees that
did not send and receive an average of at least one message per
month. Ultimately, our data set spanned 28 months (December 1999 --
March 2002), involved 105 employees, and a total of 33,079
messages. The data is arranged so that the senders are sorted by
frequency (greatest to least).
The tensor representation has a total of 8,540 nonzeros (many of the
messages occur between the same sender/receiver pair in the same time period).
The tensor is 2.7\% dense.

We apply CP-APR to find a model for the data.
There is no ideal method for choosing the number of
components. Typically, this value is selected through trial and error,
trading off accuracy (as the number of components grows) and model
simplicity.  
Here we show results for $R=10$ components.
We use the same settings for CP-APR as specified in \App{compare_objectives-app}.

\newcommand{\EnronFig}[2]{\subfloat[Component #1]{\label{fig:enron-#2}%
\includegraphics[width=0.45\textwidth,trim=10 25 10 10,clip]{component-#2}}}

\begin{figure}
  \centering
  \EnronFig{1}{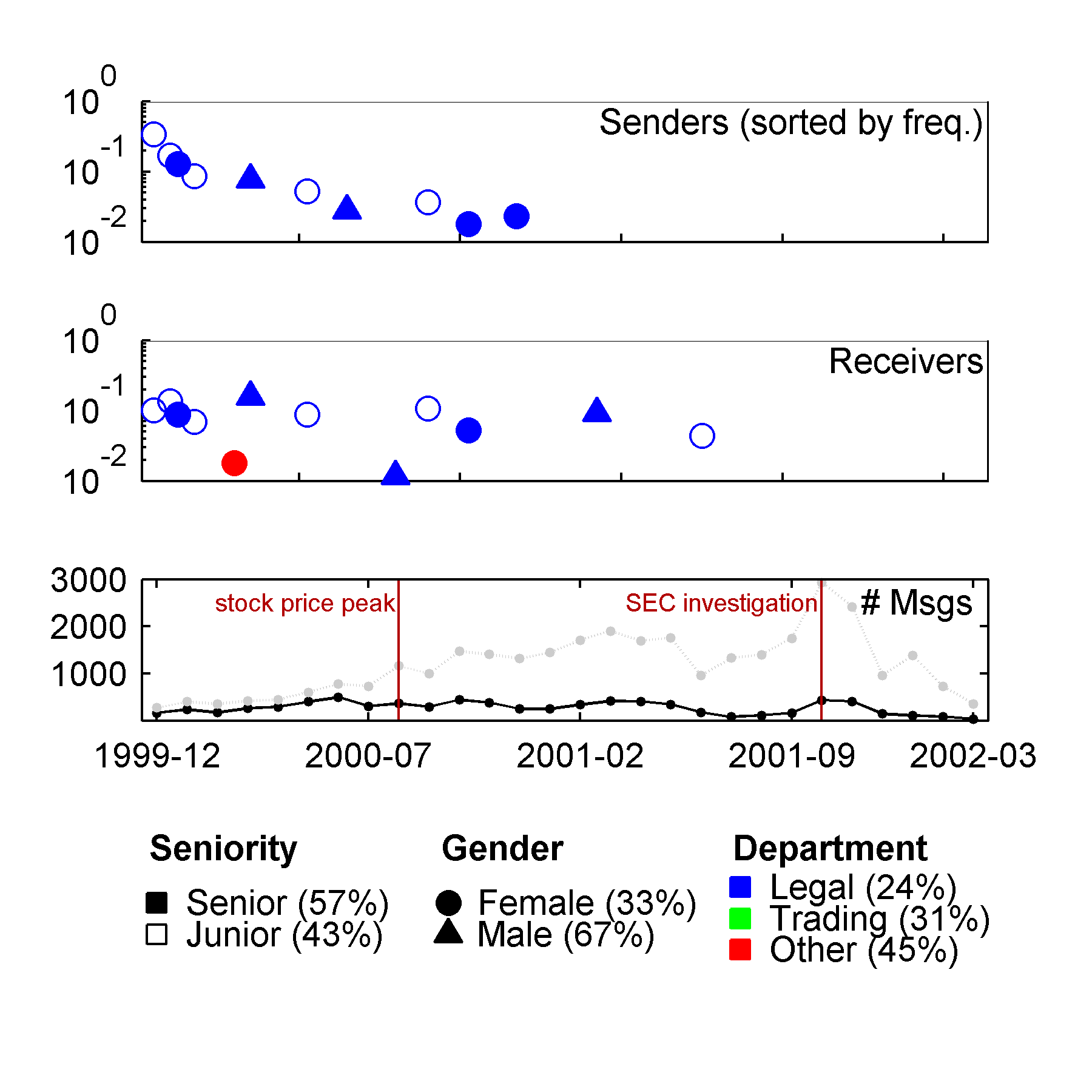}~
  \EnronFig{3}{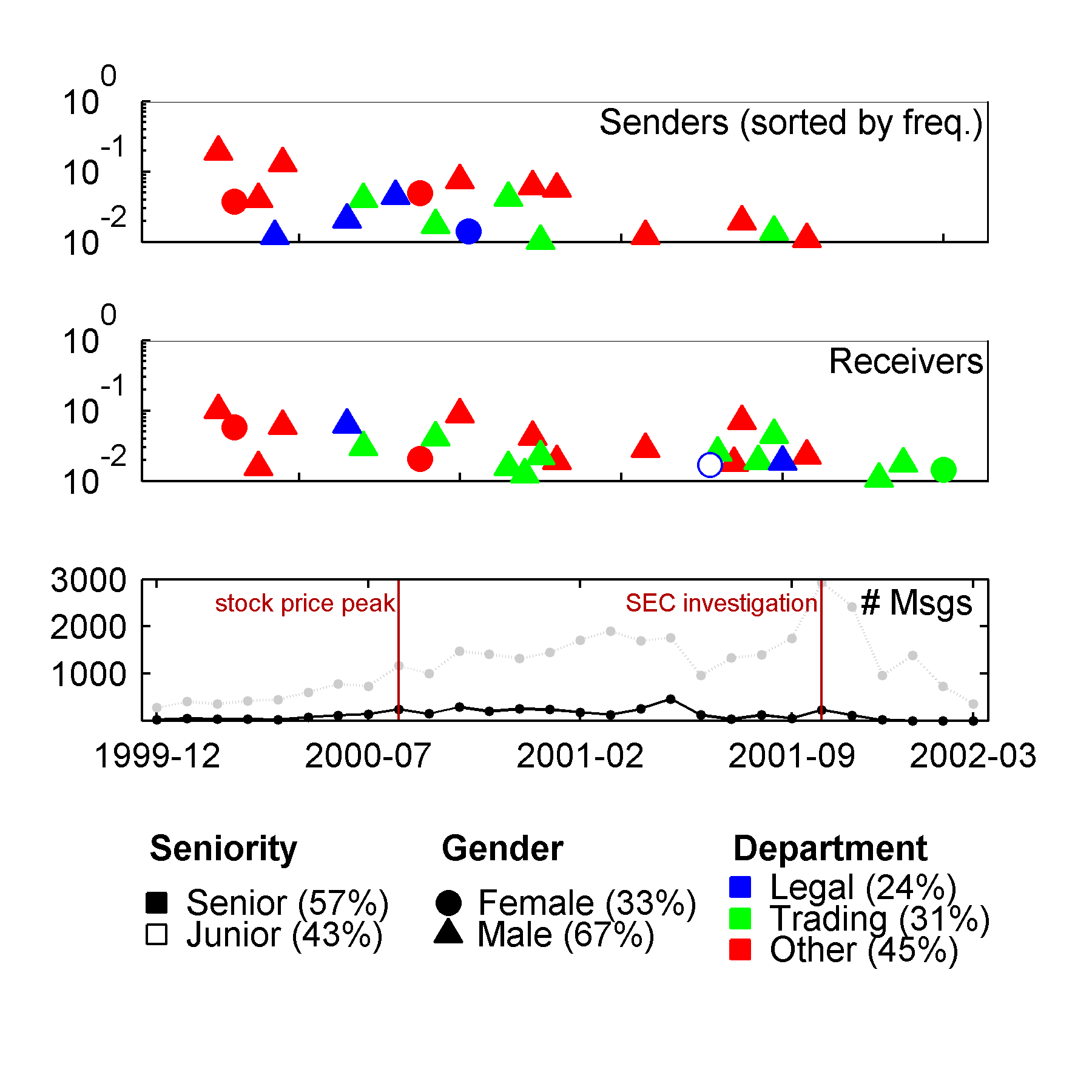}\\
  \EnronFig{4}{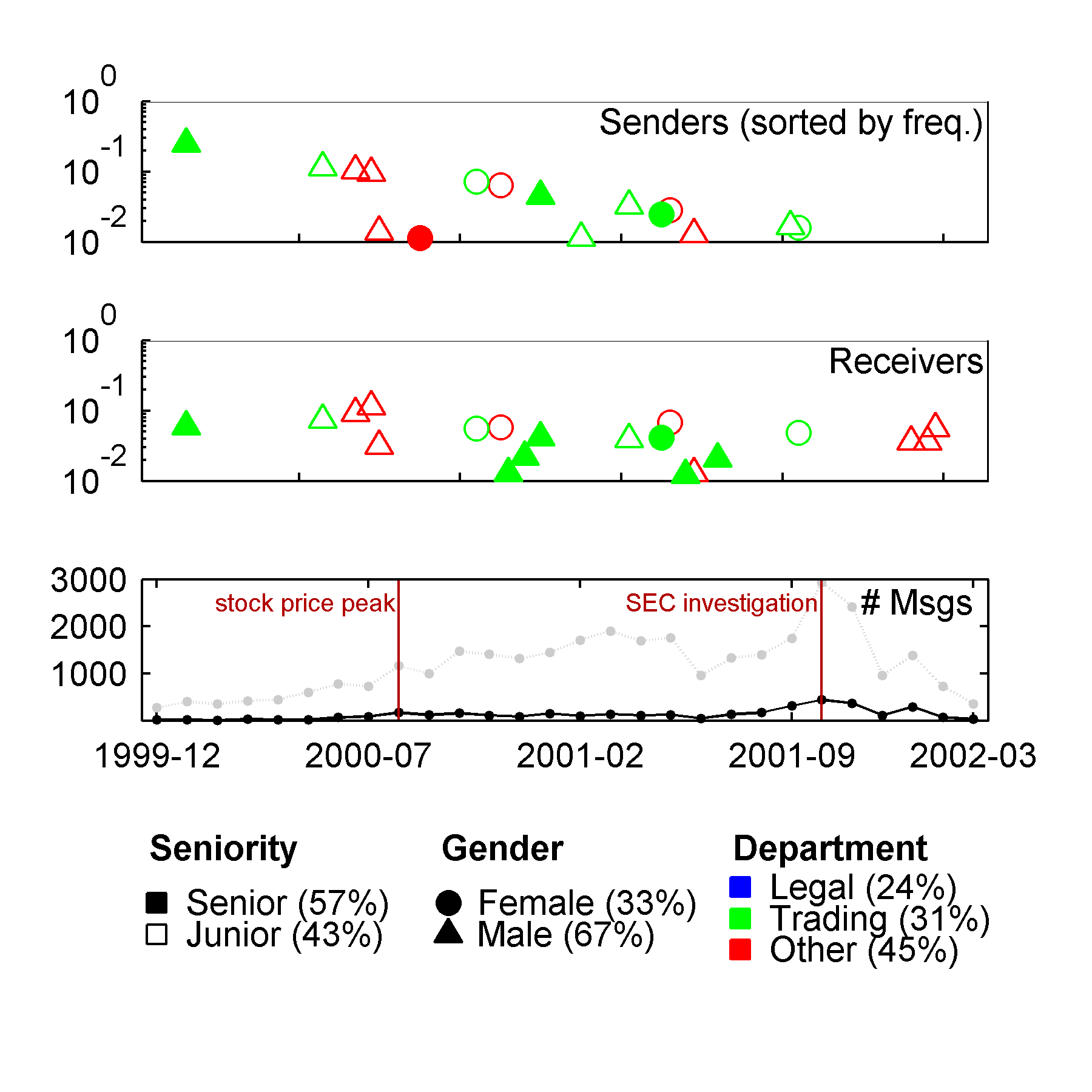}~
  \EnronFig{5}{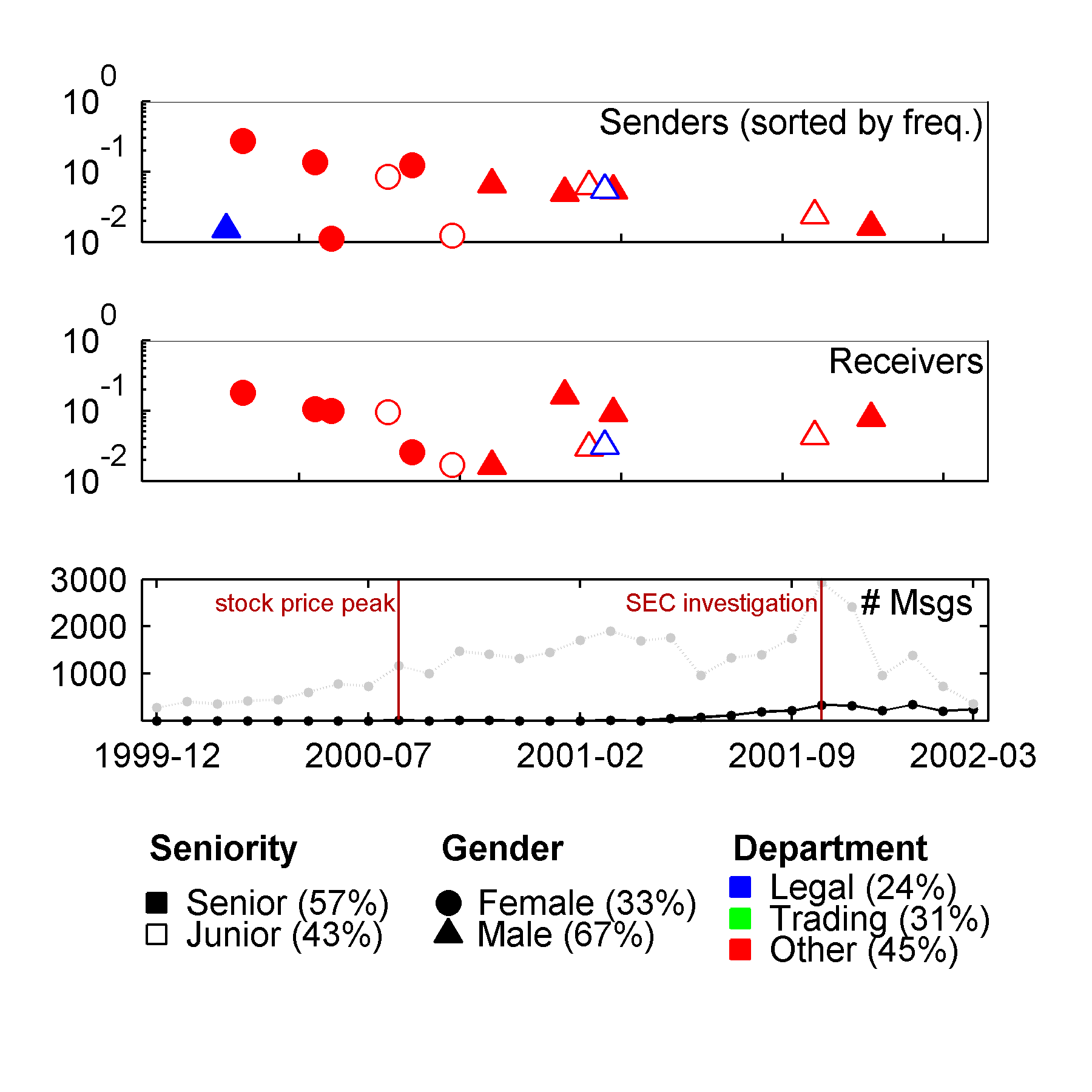}\\
  \EnronFig{8}{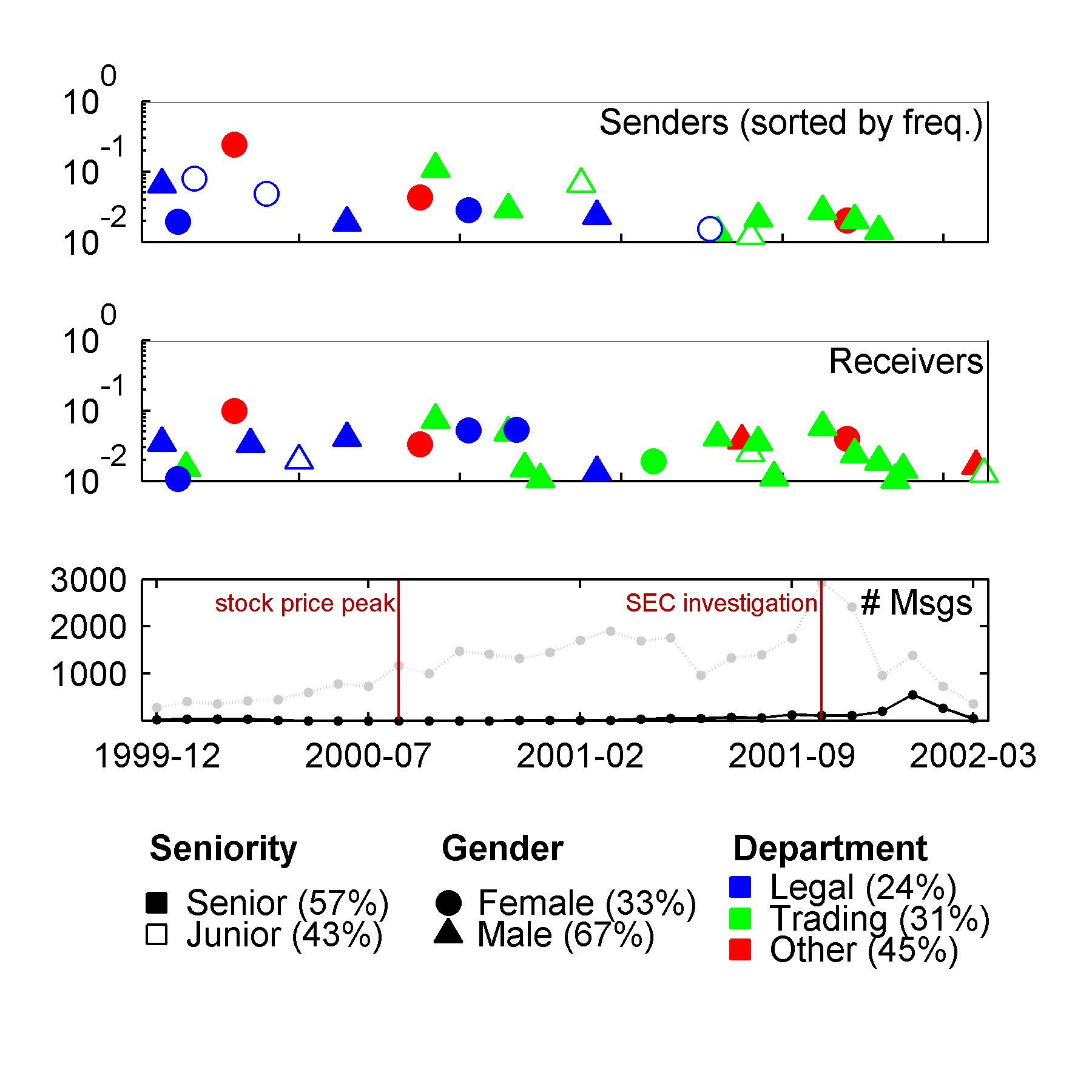}~
  \EnronFig{10}{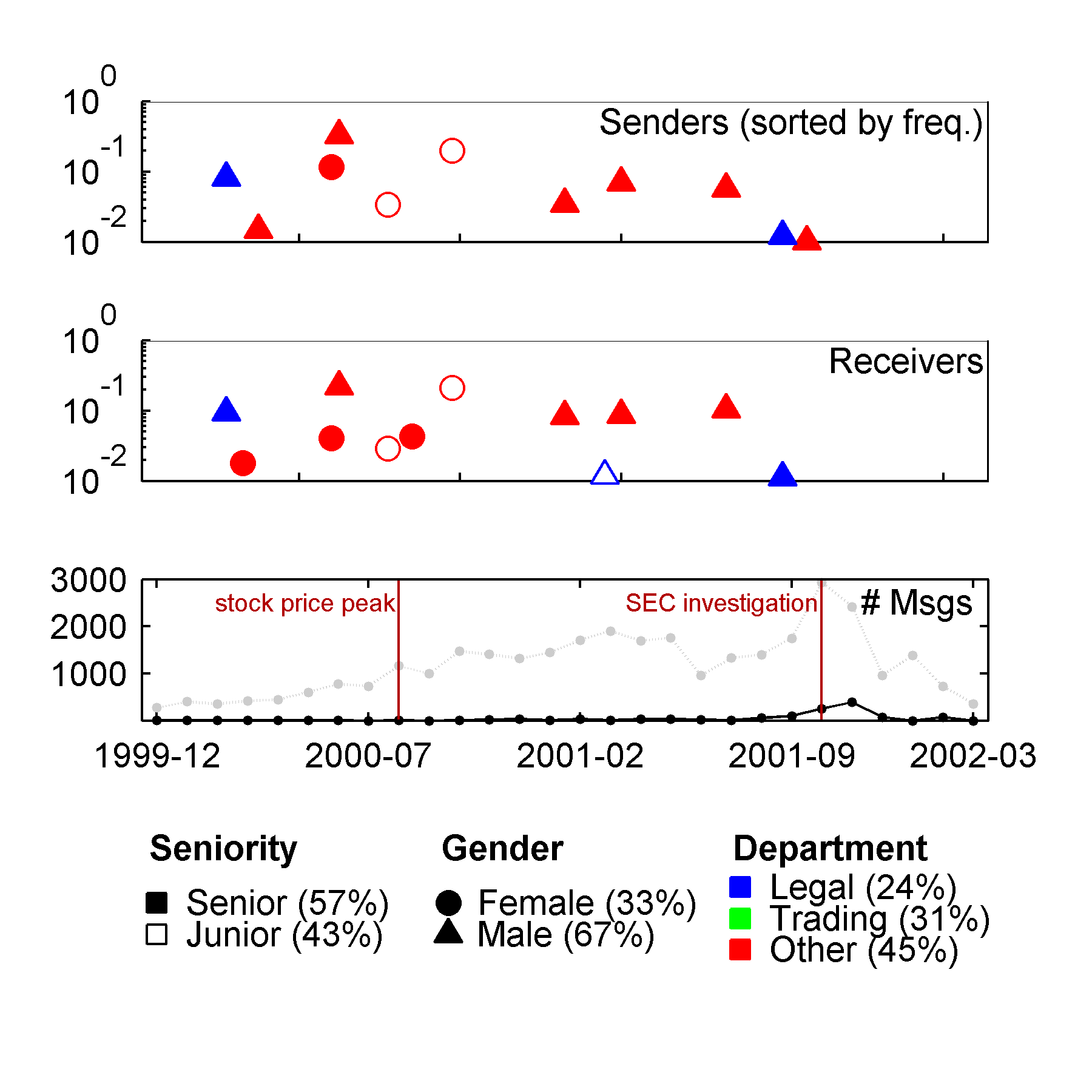}
  \caption{Components from factorizing the Enron data.}
  \label{fig:enron}
\end{figure}

\Fig{enron} illustrates six components in the resulting
factorization; the other four are shown in \App{enron-app}. For each component, the top two plots shows the
activity of senders and receivers, with the employees ordered from
left to right by frequency of sending emails. Each employee has a
symbol indicating their seniority (junior or senior), gender (male or
female), and department (legal, trading, other). The sender and
receiver factors have been normalized to sum to one, so the height of
the marker indicates each employee's relative activity within the
component. The third component (in the time dimension) is scaled so
that it indicates total message volume explained by that
component. The light gray line shows the total message volume. It is
interesting to observe how the components break down into specific
subgroups. For instance, component 1 in \Fig{enron-01} consists of
nearly all ``legal'' and is majority
female. This can be contrasted to component 5 in \Fig{enron-05}, which is nearly all
``other'' and also majority female.
Component 3 in \Fig{enron-03} is a conversation among ``senior'' staff
and mostly male; on the other hand, ``junior'' staff are more prominent in
Component 4 in \Fig{enron-04}.
Component 8 in \Fig{enron-08} seems to be a conversation among ``senior'' staff after
the SEC investigation has begun. 
Component 10 in \Fig{enron-10} indicates that a couple of ``legal''
staff are communicating with many ``other'' staff immediately after
the SEC investigation is announced, perhaps advising the ``other''
staff on appropriate responses to investigators.

\subsection{SIAM Data}
As another example, we consider five years (1999-2004) of SIAM
publication metadata that has previously been used by Dunlavy et
al.~\cite{DuKoKe11}. Here, we build a three-way sparse tensor based on
title terms (ignoring common stop words), authors, and journals. The
author names have been normalized to last name plus initial(s).  The
resulting tensor is of size 4,952 (terms) $\times$ 6,955 (authors)
$\times$ 11 (journals) and has 64,133 nonzeros (0.017\% dense).  The
highest count is 17 for the triad (`education', `Schnabel B', `SIAM
Rev.'), which is a result of Prof. Schnabel's writing brief
introductions to the education column for \emph{SIAM Review}. In fact,
the next 4 highest counts correspond to the terms `problems',
`review', `survey', and `techniques', and to authors `Flaherty J' and
`Trefethen N'.

\setlength\tymin{2in}
\begin{table}
  \centering\footnotesize
  \begin{tabulary}{\linewidth}{|C|L|L|L|}
    \hline
\bf \# & \bf Terms & \bf Authors & \bf Journals \\ \hline
1 & 
graphs, problem, algorithms, approximation, algorithm, complexity, optimal, trees, problems, bounds & 
Kao MY, Peleg D, Motwani R, Cole R, Devroye L, Goldberg LA, Buhrman H, Makino K, He X, Even G & 
SIAM J Comput, SIAM J Discrete Math\\ \hline
2 & 
method, equations, methods, problems, numerical, multigrid, finite, element, solution, systems & 
Chan TF, Saad Y, Golub GH, Vassilevski PS, Manteuffel TA, Tuma M, Mccormick SF, Russo G, Puppo G, Benzi M & 
SIAM J Sci Comput\\ \hline
3 & 
finite, methods, equations, method, element, problems, numerical, error, analysis, equation & 
Du Q, Shen J, Ainsworth M, Mccormick SF, Wang JP, Manteuffel TA, Schwab C, Ewing RE, Widlund OB, Babuska I & 
SIAM J Numer Anal\\ \hline
4 & 
control, systems, optimal, problems, stochastic, linear, nonlinear, stabilization, equations, equation & 
Zhou XY, Kushner HJ, Kunisch K, Ito K, Tang SJ, Raymond JP, Ulbrich S, Borkar VS, Altman E, Budhiraja A & 
SIAM J Control Optim\\ \hline
5 & 
equations, solutions, problem, equation, boundary, nonlinear, system, stability, model, systems & 
Wei JC, Chen XF, Frid H, Yang T, Krauskopf B, Hohage T, Seo JK, Krylov NV, Nishihara K, Friedman A & 
SIAM J Math Anal\\ \hline
6 & 
matrices, matrix, problems, systems, algorithm, linear, method, symmetric, problem, sparse & 
Higham NJ, Guo CH, Tisseur F, Zhang ZY, Johnson CR, Lin WW, Mehrmann V, Gu M, Zha HY, Golub GH & 
SIAM J Matrix Anal A\\ \hline
7 & 
optimization, problems, programming, methods, method, algorithm, nonlinear, point, semidefinite, convergence & 
Qi LQ, Tseng P, Roos C, Sun DF, Kunisch K, Ng KF, Jeyakumar V, Qi HD, Fukushima M, Kojima M & 
SIAM J Optimiz\\ \hline
8 & 
model, nonlinear, equations, solutions, dynamics, waves, diffusion, system, analysis, phase & 
Venakides S, Knessl C, Sherratt JA, Ermentrout GB, Scherzer O, Haider MA, Kaper TJ, Ward MJ, Tier C, Warne DP & 
SIAM J Appl Math\\ \hline
9 & 
equations, flow, model, problem, theory, asymptotic, models, method, analysis, singular & 
Klar A, Ammari H, Wegener R, Schuss Z, Stevens A, Velazquez JJL, Miura RM, Movchan AB, Fannjiang A, Ryzhik L & 
SIAM J Appl Math\\ \hline
10 & 
education, introduction, health, analysis, problems, matrix, method, methods, control, programming & 
Flaherty J, Trefethen N, Schnabel B, [None], Moon G, Shor PW, Babuska IM, Sauter SA, Van Dooren P, Adjei S & 
SIAM Rev\\ \hline
  \end{tabulary}
  \caption{Highest-scoring items in a 10-term factorization of the term $\times$ author $\times$ journal tensor from five years of SIAM publication data.}
  \label{tab:siam}
\end{table}

Computing a ten-component factorization yields the results shown in
\Tab{siam}.  
We use the same settings for CP-APR as specified in \App{compare_objectives-app}.
In the table, for the term and author modes, we list any entry whose
factor score is greater than $10^{-7}\cdot I_n$, where $I_n$ is the
size of the $n$th mode; in the journal mode, we list any entry greater
than 0.01.
The 10th component corresponds to introductions written by section
editors for \emph{SIAM Review}. The 1st component shows that there is
overlap in both authors and title keyword between the \emph{SIAM
  J. Computing} and the \emph{SIAM J. Discrete Math}. The 2nd and 3rd
components have some overlap in topic and two overlapping authors, but
different journals. Both components 8 and 9 correspond to the same
journal but reveal two subgroups of authors writing on slightly
different topics.

\section{Conclusions \& Future Work}
\label{sec:conclusions}

We have developed an alternating Poisson regression fitting algorithm,
CP-APR, for PTF for sparse count data. 
When such data is generated via a Poisson process, we show that 
methods based on KL divergence such as CP-APR recovers the true CP
model more reliably than methods based on LS. Indeed, 
in classical statistics, it is well known that the randomness observed
in sparse count data is better explained and analyzed by the Poisson
model (KL divergence) than a Gaussian one (LS error). 

Our algorithm can be considered an extension of the Lee-Seung method
for KL divergence with multiple inner iterations (similar to \cite{GiGl11}
for LS). 
Allowing for multiple inner iterations  has the benefit of
accelerating convergence.  
Moreover, being very similar to an existing method, CP-APR is simple
to implement with the exception of some details of the sparse implementation as
described in \Sec{sparse}. 
To the best of our knowledge, ours is the
first implementation of any KL-divergence-based method
for large-scale sparse tensors.

In \Sec{convergence}, we provide a general-purpose convergence proof
for the alternating Gauss-Seidel approach. 
The regularity conditions imposed in our proofs make rigorous and
concrete our intuition that in the context of sparse count data,
CP-APR will converge provided that the data tensor meets a minimal
density and that nonzeros are sufficiently spread throughout the data
tensor with respect to the size of the factor matrices being fit. 
Any subproblem solver can
be substituted for the MM method without changing the theory.  A
benefit of the MM subproblem solver is 
that its multiplier matrix can be used to
explicitly track convergence based on the KKT conditions. Moreover, we
observe that we can use the KKT information to identify and correct
inadmissible zeros using a ``scooch.'' Lin \cite{Li07a} had a similar
observation in the LS case but came up with a different
correction technique. We analyze convergence of the MM subproblem with
the ``scooch'' in order to show that it will always converge. Our
results are stronger than past results because they allow iterates
with some zero entries. Even though zero entries are possible to avoid
in exact arithmetic, they often occur numerical computations and so
are important to consider.

There remains much room for future work. Foremost among practical
considerations is speed of convergence. Although multiplicative
updates are relatively simple to compute, CP-APR can require many
iterates. One approach to accelerating convergence would be to replace
the MM algorithm subproblem solver. For example, Kim et al.\@
\cite{KiSrDh10} present fast quasi-Newton methods for minimizing
box-constrained 
convex functions that can be used to solve a nonnegative LS
or minimum KL-divergence subproblem in a nonlinear Gauss-Seidel
solver.  
A second approach is to focus on the sequence of outer iterates. 
Zhou et al.\@ \cite{ZhAlLa11} provide a general quasi-Newton
acceleration scheme for iterative methods based on a quadratic
approximation of the iteration map instead of the loss. 

There has also been significant work in finding sparse factors via
$\ell_1$-penalization for matrices \cite{LiZhJiSh10} and tensors
\cite{MoHaAr08, WaMaLoLi08, FrHa08, LiLiWoYe12}. Sparse factors often
provide more easily interpreted models, and penalization may also
accelerate the convergence. While the factor matrices generated by
CP-APR may be naturally sparse without imposing an $\ell_1$-penalty, the
degree of sparsity is not currently tunable.
One may also consider extensions of this work in the context of
missing data \cite{Ki97,BuFi05,ToBr05,AcDuKoMo11} and for alternative
tensor factorizations such as Tucker \cite{FrHa08}.

Perhaps most challenging, however, are open questions related to rank and inference. Questions about how to choose rank are not new; but given the context of sparse count data, might that structure
be exploited to derive a sensible heuristic or even rigorous criterion for choosing the rank? We already see that \As{full_row_rank} imposes an upper bound on the rank to ensure algorithmic convergence.  Regarding inference, our focus in this work was in thoroughly developing the algorithmic groundwork for fitting a PTF model for sparse count data. CP-APR can be used to estimate latent structure. Once an estimate is in hand, however, it is natural to ask how much uncertainty there is in that estimate. For example, is it possible to put a confidence interval around the entries in the fitted factor matrices, especially zero or near zero entries? Given that inference for the related but simpler case of Poisson regression has been worked out, we suspect that a sensible solution is waiting to be found.  The benefits of answering these questions warrant further investigation. We highlight them as important topics for future research. 

\section*{Acknowledgments}
We thank our colleagues at Sandia for numerous helpful conversations
in the course of this work, especially Grey Ballard and Todd
Plantenga. We also thank Kenneth Lange for pointing us to relevant references
on emission tomography. Finally, we thank the anonymous referees and associate editor for 
suggestions which greatly improved the quality of the manuscript.

\appendix
\section{Notation Details}
\label{sec:notation-details}

\paragraph{Outer product} The outer product of $N$
vectors is an $N$-way tensor. For example,
\begin{inlinemath}
  (\V{a} \circ \V{b} \circ \V{c})_{ijk} = \VE{a}{i} \VE{b}{j} \VE{c}{k}.
\end{inlinemath}

\paragraph{Elementwise multiplication and division}

Let $\T{A}$ and $\T{B}$ be two same-sized tensors (or matrices). Then
$\T{C} = \T{A} \Hada \T{B}$ yields a tensor that is the same size as
$\T{A}$ (and $\T{B}$) such that  $\TE{c}{i} = \TE{A}{i} \TE{B}{i}$ for
all $\MI{i}$. Likewise,
$\T{C} = \T{A} \Divide \T{B}$ yields a tensor that is the same size as
$\T{A}$ (and $\T{B}$) such that  $\TE{c}{i} = \TE{A}{i} / \TE{B}{i}$ for
all $\MI{i}$.

\paragraph{Khatri-Rao product}
Give two matrices $\M{A}$ and $\M{B}$ of sizes $I_1 \times R$ and $I_2
\times R$, then $\M{C} = \M{A} \Khat \M{B}$ is a matrix of size $I_1
I_2 \times R$ such that
\begin{displaymath}
  \M{C} =
  \begin{bmatrix}
    \MC{A}{1} \Kron \MC{B}{1}
    & \MC{A}{2} \Kron \MC{B}{2}
    & \cdots
    & \MC{A}{R} \Kron \MC{B}{R}
  \end{bmatrix},
\end{displaymath}
where the Kronecker product of two vectors of size $I_1$ and $I_2$ is
a vector of length $I_1I_2$ given by
\begin{displaymath}
  \V{a} \Kron \V{b} =
  \begin{bmatrix}
    a_1 \V{b} \\ a_2 \V{b} \\ \vdots \\ a_{I_1} \V{b}
  \end{bmatrix}.
\end{displaymath}

\paragraph{Matricization of a tensor}
The mode-$n$ matricization or unfolding of a tensor
$\T{X}$ is denoted by $\Mz{X}{n}$ and is of size $I_n \times J_n$
where $J_n \equiv \prod_{m \neq n} I_n$. In this case, tensor element
$\MI{i}$ maps to matrix element $(i,j)$ where
\begin{displaymath}
  i = i_n \qtext{ and }
  j = 1 + \displaystyle\sum_{{k=1}\atop{k \neq n}}^N (i_k - 1) \left( \prod_{{m=1}\atop{m \neq
        n}}^{k-1} I_m \right).
\end{displaymath}

\section{Proof of \Lem{OmegaZeta}}
\label{sec:proof-OmegaZeta}
In this section, we provide a proof for \Lem{OmegaZeta}. We first
establish two useful lemmas.

\smallskip
\begin{lemma}\label{lem:lambda-bound}
  Let $\T{X}$ be fixed, let 
  $\TM =\KTsmall{\Vl; \Mn{A}{1},\dots,\Mn{A}{N}}$, and
  let $f(\TM)$ be the objective function as in \Eqn{nlp}.
  If $f(\TM) \leq \zeta$ for some constant $\zeta > 0$, then
  there exists constants $\xi',\xi > 0$ (depending on $\T{X}$ and $\zeta$) such that 
  $\V{e}\Tra \Vl \in [\xi', \xi]$. 
\end{lemma}
\begin{proof}
  Because the factor matrices are column stochastic, we can observe that
  \begin{equation}\label{eq:f-bound}
    \begin{aligned}
      f(\T{M})
      & = \V{e}\Tra\V{\lambda} - \sum_{\V{i}} \TE{X}{\V{i}}
      \log \left (\sum_{r} \lambda_r \; \MnE{A}{n}{i_1 r}
        \cdots \MnE{A}{n}{i_N r} \right ), \\
      & \geq \V{e}\Tra\Vl -
      \vartheta \log \left (\V{e}\Tra\Vl \right )
      \qtext{where} \vartheta = \left (\prod_{n=1}^N I_n \right ) \max_{\MI{i}}
      \, \TE{X}{i}.
    \end{aligned}
  \end{equation}
  We have $\zeta \geq \V{e}\Tra\Vl - \vartheta \log \left
    (\V{e}\Tra\Vl \right )$.  Let
    $g(\alpha) = \alpha -
  \vartheta \log(\alpha)$ where $\alpha >
  0$. We show that $g(\alpha) \leq \zeta$ implies there exists $\xi',\xi > 0$ such
  that $\alpha \in [\xi',\xi]$.  First assume there is no such lower bound $\xi'$.
    Then there is a sequence $\alpha_n$ tending to zero such that $g(\alpha_n) \leq \zeta$.
    But for sufficiently large $n$, we have that $-\vartheta \log(\alpha_n) > \zeta$. Since $\alpha_n > 0$
    for all $n$, we have that for sufficiently large $n$ the function $g(\alpha_n) > \zeta$. Therefore,
    there is such a lower bound $\xi'$.
    
    Now suppose there is no such upper bound $\xi$, and therefore there is an unbounded and increasing sequence
    $\alpha_n$ tending to infinity such that $g(\alpha_n) \leq \zeta$ for all $n$.
     Note that $g'(\alpha) = 1 - \vartheta/\alpha$.
    Since $g(\alpha)$ is convex, we have that
    \begin{equation*}
    g(\alpha) \geq g(2\vartheta) + g'(2\vartheta)(\alpha - 2\vartheta) = g(2\vartheta) + \frac{1}{2}\alpha - \vartheta.
    \end{equation*}
	This inequality, however, indicates that for sufficiently large $n$, the right hand side is greater than $\zeta$.
	Therefore, there must be an upper bound $\xi$.
  Substituting $\alpha = \V{e}\Tra\Vl$ completes the proof.
\end{proof}

\smallskip
\begin{lemma}\label{lem:finiteOnOmegaZeta}
  Let $\T{X}$ be fixed, and
  let $f(\TM)$ be the objective function as in \Eqn{nlp}.
  Let $\Omega(\zeta)$ be the convex hull of the level set of $f$ as defined in \Eqn{OmegaZeta}.
  The function $f(\TM)$ is bounded for all $\TM \in \Omega(\zeta)$.
\end{lemma}
\begin{myproof}
  Let $\Tbar{M},\That{M} \in \Set{\T{M} | f(\TM) \leq \zeta}$.
  Define
  $\Ttilde{M}$ to be the convex combination
  \begin{displaymath}
    \Ttilde{M} = \alpha \Tbar{M} + (1-\alpha) \That{M}
    \qtext{ where } \alpha \in [0.5,1).
  \end{displaymath}
  Note that the restriction on $\alpha$ is arbitrary but makes the proof
  simpler later on.
  Observe that
  \begin{displaymath}
    \tilde m_{\MI{i}} =
    \sum_r \left\{
      \left( \alpha \bar \lambda_r + (1 - \alpha) \hat \lambda_r \right)
      \prod_{n} \left(  \alpha \bar a_{i_nr}^{(n)} + (1 - \alpha) \hat a_{i_nr}^{(n)} \right)
    \right\}
  \end{displaymath}
  On the one hand, by \Lem{lambda-bound}, there exists $\xi > 0$ such that
  \begin{displaymath}
    \tilde m_{\MI{i}}
    \leq \sum_r \left( \alpha \bar \lambda_r + (1 - \alpha) \hat \lambda_r \right)
    = \alpha \sum_r \bar \lambda_r + (1-\alpha) \sum_r  \hat \lambda_r
    \leq \alpha \xi + (1 - \alpha) \xi = \xi.
  \end{displaymath}
  On the other hand,
  \begin{displaymath}
    \tilde m_{\MI{i}}
    \geq \sum_r \left\{
      \alpha \bar \lambda_r
      \prod_{n} \alpha \bar a_{i_nr}^{(n)}
    \right\}
    = \alpha^{N+1} \bar m_{\MI{i}}
  \end{displaymath}
  Thus,
  \begin{displaymath}
    \alpha^{N+1} \bar m_{\MI{i}}
    \leq \tilde m_{\MI{i}}
    \leq \bar m_{\MI{i}}  + \xi
  \end{displaymath}
  Now consider
  \begin{align*}
    \tilde m_{\MI{i}}  - \TE{x}{i} \log \tilde m_{\MI{i}}
    &\leq \bar m_{\MI{i}}  + \xi - \TE{x}{i} \log \alpha^{N+1} \bar  m_{\MI{i}} \\
    &=  \left( \bar m_{\MI{i}}  - \TE{x}{i} \log \bar m_{\MI{i}} \right)
    + \xi - (N+1) \TE{x}{i} \log \alpha \\
    &\leq \left( \bar m_{\MI{i}}  - \TE{x}{i} \log \bar m_{\MI{i}} \right)
    + \xi + (N+1)\TE{x}{i} \log 2.
  \end{align*}
  Thus,
  \begin{displaymath}
    f(\Ttilde{M})
    \leq f(\Tbar{M}) + \xi \prod_n I_n + (N+1)\log 2 \sum_i \TE{x}{i}
    \leq \xi \left( 1 + \prod_n I_n \right) + (N+1)\log 2 \sum_i  \TE{x}{i}.
    \myproofend
  \end{displaymath}
\end{myproof}

Given these two lemmas, we are finally ready to provide the proof of
\Lem{OmegaZeta}.

\smallskip
\begin{proof}[of \Lem{OmegaZeta}]
  Fix $\zeta$. If $\Set{ \TM \in \Omega | f(\TM)
    \leq \zeta }$ is empty, then $\Omega(\zeta)$ is empty and there
  is nothing left to do. Thus, assume $\Set{ \TM \in \Omega | f(\TM)
    \leq \zeta }$ is nonempty.
  Since $f$ is continuous at all $\T{M} \in \Omega$ for which
  $f(\T{M})$ is finite, $f$ is obviously continuous on
  $\Omega(\zeta)$ by \Lem{finiteOnOmegaZeta}. Since $f$ is continuous, $\Set{ \TM \in \Omega |
    f(\TM) \leq \zeta }$ is closed because it is the preimage of the
  closed set $(-\infty, \zeta]$ under $f$; thus,
  $\Omega(\zeta)$ is closed because it is a convex combination of
  closed sets. Consequently, we only need
  to show that  $\Omega(\zeta)$ is bounded.
  Assume the contrary.
  Then there exists a sequence of models $\T{M}\It{k} =
  \KT{\V{\lambda}\It{k};\Mn{A}{1}\It{k},\dots,\Mn{A}{N}\It{k}} \in
  \Omega(\zeta)$
  such that $\V{e}\Tra\V{\lambda}\It{k}
  \rightarrow \infty$. By \Lem{finiteOnOmegaZeta}, $f(\T{M})$ is
  finite on $\Omega(\zeta)$, but this contradicts \Lem{lambda-bound}.
  Hence, the claim.
\end{proof}

\section{Deriving the MM updates}
\label{sec:proof-mm-unique-global-min}

In this section we derive the MM update rules used to solve the subproblem.
We first verify that \Eqn{majorization}
majorizes \Eqn{subproblem-simple}.
For convenience let $\M{C} = \MB\Tra$ so that
\Eqn{subproblem-simple} reduces to
\begin{equation}\label{eq:subproblem_redux}
  \min_{\M{C} \geq 0} f(\M{C}\Tra) =
  \sum_{ij}  \MC{C}{i}\Tra\MC{\pi}{j} - x_{ij}
  \log \left( \MC{C}{i}\Tra\MC{\pi}{j} \right)
\end{equation}

Proofs of the next two lemmas are given by Lee and Seung in \cite{LeSe01} but their arguments
do not carefully handle boundary points. The following two lemmas and their proofs treat with more rigor
the existence and value of updates when anchor points lie on admissible regions of the boundary.

\begin{lemma}\label{lem:mm}
  Let $x \geq 0$ be a scalar and $\V{\pi} \geq 0$, $\V{\pi} \neq 0$,
  be a vector of length $R$. For a vector $\V{c} \geq 0$, $\V{c} \neq
  0$, of length $R$, let the function $f$ be defined by
  \begin{displaymath}
    f(\V{c}) = \V{c}\Tra\V{\pi} - x \log \left( \V{c}\Tra\V{\pi} \right).
  \end{displaymath}
  Then $f$ is majorized at $\Vbar{c} \geq 0$ by
  \begin{displaymath}
    g(\V{c}, \Vbar{c}) = \V{c}\Tra\V{\pi} - x \sum_{r=1}^R \alpha_r
    \log \left( \frac{c_r\pi_r}{\alpha_r} \right)
    \qtext{where}
    \alpha_r = \frac{\bar c_r\pi_r}{\Vbar{c}\Tra\V{\pi}}.
  \end{displaymath}
\end{lemma}

\label{sec:proof-mm}
\begin{proof}
  If $x  = 0$, then $g(\V{c}, \Vbar{c}) = f(\V{c})$ for all $\V{c}$, and $g$ trivially majorizes $f$ at $\Vbar{c}$.
  Consider the case when $x > 0$.
  It is immediate that $g(\Vbar{c}, \Vbar{c}) = f(\Vbar{c})$.
  The majorization follows from the fact that $\log$ is strictly
  concave and that we can write $\V{c}\Tra\V{\pi}$ as a convex
  combination of the elements $c_r \pi_r / \alpha_r$.
  Note that if any elements $\bar c_r \pi_r$ are zero, they do not
  contribute to the sum since we assume %
  $0 \cdot \log(\mu) = 0$ for all $\mu \geq 0$ and $\alpha_r = 0$.
\end{proof}

We now derive an expression for the unique global minimizer
of majorization. The majorization defined in \Eqn{majorization} can be expressed in terms of $\M{C}$ as

\begin{equation}
   \label{eq:majorization_redux}
  g(\M{C}, \Mbar{C}) = \sum_{rij}
  \left [ {c}_{ri} {\pi}_{rj} -
  \alpha_{rij} x_{ij} \log\left( \frac{c_{ri} \pi_{rj}}{\alpha_{rij}}
    \right) \right]
  \qtext{where}
  \alpha_{rij} = \frac{\bar c_{ri}\pi_{rj}}{\sum_r \bar c_{ri} \pi_{rj}}.  
\end{equation}

\begin{lemma}\label{lem:mm-unique-global-min}
  Let $f$ and $g$ be as defined in \Eqn{subproblem_redux} and
  \Eqn{majorization}, respectively. Then, for all $\Mbar{C} \geq
  \M{0}$ such that $f(\Mbar{C}\Tra)$ is finite, the function $g(\cdot,
  \Mbar{C})$ has a unique global minimum $\M{C}\It{*}$ which is given
  by
  $(\M{C}\It{*})_{ri} = \sum_j  \alpha_{rij} x_{ij}$
  where
  $\alpha_{rij} = {\bar
      c_{ri}\pi_{rj}}/{\MbarC{c}{i}\Tra\MC{\pi}{j}}$,
    for all
    $r = 1,\dots,R$, $i = 1,\dots,I$.
\end{lemma}

\begin{proof}
  Because $g(\M{C}, \Mbar{C})$ separates in the elements of $\M{C}$ we
  focus on solving each elementwise minimization problem.  Dropping
  subscripts, the minimization problem with respect to $c_{ri}$ can be
  rewritten as
  \begin{equation}\label{eq:mm-univariate-global-min}
    \min_{c \geq 0} \; c - \sum_j  \alpha_j x_j \log \left (\frac{c \pi_j}{\alpha_j} \right),
  \end{equation}
  where we have used the fact that $\sum_j \pi_j = 1$. It is
  sufficient to prove that this univariate problem has a unique global
  minimizer, $c\It{*} = \sum_j \alpha_j x_j$.
  First, consider the case where the second term is nonzero. Inspecting the stationarity condition
  reveals the solution. Moreover, the function is
  strictly convex and so has a unique global minimum.  Second,
  consider the case where the second term is zero. Then, it is
  immediate that the unique global minimum is $c\It{*} = 0$. Moreover,
  the second term can only vanish when $\sum_j \alpha_j x_j = 0$, and
  so the formula applies.
\end{proof}

\section{Proof of \Thm{subproblem-convergence}}
\label{sec:subproblem-convergence}
In this section, we prove the MM Algorithm in \Alg{CPAPR} solves
\Eqn{subproblem-simple}. We first establish the following general
result for algorithm maps. Part \Part{limit_points} is a simple version of Zangwill's convergence theorem \cite[p.~91]{Zangwill1969} in the
case where the objective function and algorithm map are both continuous. The proof of part \Part{successive_iterates}
follows arguments of part of a proof for a different but related property on MM iterates in \cite[p.~198]{La04b}.

\begin{theorem}\label{thm:MM_limit_points}
  Let $f$ be a continuous function on a domain $\mathcal{D}$, and
  let $\psi$ be a continuous iterative map from $\mathcal{D}$ into $\mathcal{D}$
  such that $f(\psi(\Vx)) < f(\Vx)$ for all $\Vx \in \mathcal{D}$ with $\psi(\Vx) \neq \Vx$.
  Suppose there is an $\Vx\It{0}$ such that the set
  $\LS{f}{\Vx\It{0}} \equiv \Set{\V{x} \in \mathcal{D} | f(\Vx) \leq f(\Vx\It{0}) }$ is compact.
  Define $\Vx\It{k+1} = \psi(\Vx\It{k})$ for $k = 0, 1, \ldots$. Then
  \begin{inparaenum}[(a)]
  \item \label{part:limit_points}
    the sequence of iterates $\{\Vx\It{k}\}$ has at least one limit point
    and all its limit points are fixed points of $\psi$, and
  \item \label{part:successive_iterates}
    the distance between successive iterates converges to 0,
    i.e., $\lVert \Vx\It{k+1} - \Vx\It{k} \rVert \rightarrow 0$.
  \end{inparaenum}
\end{theorem}

\begin{proof}
  The proof of \Part{limit_points} follows that of Proposition 10.3.2 of \cite{La04b}.
  First note that the sequence of iterates must be in $\LS{f}{\Vx\It{0}}$ because
  $f(\Vx\It{k}) \leq f(\Vx\It{0})$ for all $k$.
  Since $\LS{f}{\Vx\It{0}}$ is compact,
  $\{\Vx\It{k}\}$ has a convergent subsequence whose limit is in  $\LS{f}{\Vx\It{0}}$; denote this as
  $\Vx\It{\kl} \rightarrow \Vx\It{*}$ as $\ell \rightarrow \infty$.
  Since $f$ is assumed to be continuous, $\lim f(\Vx\It{\kl}) = f(\Vx\It{*})$.
  Moreover, clearly $f(\Vx\It{*}) \leq f(\Vx\It{\kl})$ for all $\kl$.

  Note that $f(\psi(\Vx\It{\kl})) \leq f(\Vx\It{\kl})$. Taking the limit
  of both sides and applying the continuity of $\psi$ and $f$, we must
  have that $f(\psi(\Vx\It{*})) \leq f(\Vx\It{*})$.  But we also have
  that
  \begin{equation*}
    f(\Vx\It{*}) \leq f(\Vx\It{k_{\ell+1}}) \leq f(\Vx\It{\kl + 1}) = f(\psi(\Vx\It{\kl})).
  \end{equation*}
  Again taking limits we obtain $f(\Vx\It{*}) \leq
  f(\psi(\Vx\It{*}))$. Therefore $f(\Vx\It{*}) =
  f(\psi(\Vx\It{*}))$. But by assumption, this equality implies that
  $\Vx\It{*}$ is a fixed point of $\psi$, and thus \Part{limit_points}
  is proven.

  We now turn to the proof of \Part{successive_iterates}, which
  follows the proof of Proposition 10.3.3 in \cite{La04b}.
  Recall $\{\Vx\It{k}\}$ denotes the iterate sequence.
  Since $f(\Vx\It{k})$ is decreasing and $f$ is bounded below on $\LS{f}{\Vx\It{0}}$,
  we can assert that $f(\Vx\It{k})$ is a convergent sequence with a limit $f\It{*}$.
  Assume the contrary of \Part{successive_iterates}, i.e., that there exists an
  $\epsilon > 0$ and a subsequence $\{\kl\}$ of the indices such that
  \begin{equation}
    \label{eq:xkl}
    \| \Vx\It{\kl + 1} - \Vx\It{\kl} \| > \epsilon
    \text{ for all } \kl.
  \end{equation}
  Note that this subsequence is different from the one discussed in
  proving part~\Part{limit_points}.
  Since $\Vx\It{\kl} \in \LS{f}{\Vx\It{0}}$, by
  possibly restricting $\{\kl\}$ to a further subsequence, we may
  assume that $\Vx\It{\kl}$ converges to a limit $\V{u}$. By possibly
  restricting $\{\kl\}$ to yet a further subsequence, we may
  additionally assume that $\Vx\It{\kl + 1}$ converges to a limit
  $\V{v}$. By \Eqn{xkl}, we can conclude $\|\V{v} - \V{u}\| \geq \epsilon$.
  Note that $\Vx\It{\kl + 1} = \psi(\Vx\It{\kl})$.
  Taking the limit of both sides and using the continuity
  of $\psi$ we obtain $\psi(\V{u}) = \V{v}$.
  Additionally, using the continuity of $f$,
  \begin{equation*}
    f(\V{u}) = \lim_{\ell\rightarrow\infty} f(\Vx\It{\kl}) = f\It{*} =  \lim_{\ell\rightarrow\infty} f(\Vx\It{\kl+1}) = f(\V{v}).
  \end{equation*}
  Since $\V{v} = \psi(\V{u})$, we have that $f(\V{u}) = f(\psi(\V{u}))$ which by assumption occurs if and only if
  $\V{u} = \psi(\V{u})$. This implies that $\V{u} = \V{v}$, and we have arrived at a contradiction.
\end{proof}

We now prove a series of lemmas leading up to a proof of the desired convergence result.

\begin{lemma}
  \label{lem:strict-decrease}
  Let $\MB \geq 0$ such that $f(\MB)$ is finite and suppose $\MB \not
  = \MB \Hada \MPhi$. Then $f(\MB) > f(\MB \Hada \MPhi)$.
\end{lemma}
\begin{proof}
  By \Lem{mm-unique-global-min} $(\MB \Hada \MPhi)\Tra$ is the unique
  global minimizer of $g(\cdot, \MB\Tra)$ which majorizes
  $f$ at $\MB\Tra$. Therefore, if $\MB \neq \MB \Hada \MPhi$, we must have
  $f(\MB) = g(\MB\Tra, \MB\Tra) > g((\MB \Hada \MPhi)\Tra, \MB\Tra)
    \geq f(\MB \Hada \MPhi)$.
\end{proof}

\begin{lemma}\label{lem:level-set-compact}
  Let $f$ be as defined in \Eqn{subproblem-simple}.  
  For any nonnegative matrix $\MB\It{0}$  such that $f(\MB\It{0})$ is
  finite, the level set $\mathcal{L}_f(\MB\It{0}) =
  \Set{\MB \geq 0 | f(\MB) \leq f(\MB\It{0}) }$ is compact.
\end{lemma}

\begin{proof}
  The proof follows the same logic as the proof for \Lem{lambda-bound}.
\end{proof}

\begin{lemma}\label{lem:MM_iterates_converge}
  Let $f$ be as defined in \Eqn{subproblem-simple} and $\psi$ be as
  defined in \Eqn{mm-iterate}, and suppose
  \As{full_row_rank_mm} is satisfied. For any nonnegative matrix $\MB\It{k}$
  such that $f(\MB\It{0})$ is finite, the sequence 
  $\MB\It{k+1} = \psi(\MB\It{k})$ converges.
\end{lemma}

\begin{proof}
  Note that all limit points of $\psi$ are fixed points of $f$ by
  \Thm{MM_limit_points}.  

  First, we show that the set of fixed point is finite.
  Suppose that $\MB$ is a fixed point of $\psi$. Then we must have
  \begin{inlinemath}
    \MB \Hada \left( \M{E} - \MPhi(\MB) \right) = 0.
  \end{inlinemath}
  By \Lem{strict-convexity}, 
  it can be verified that $\MB$ is the \emph{unique}
  global minimizer of
  \begin{displaymath}
    \min \text{~} f(\M{U}) \quad \text{s.t.~} \M{U} \in \Set{ \M{U}
      \geq 0 | \ME{u}{ir} = 0 \text{ if } \ME{b}{ir} = 0},
  \end{displaymath}
  where $f$ is as defined in \Eqn{subproblem-simple}.
  Therefore, any fixed point that has the same zero pattern of $\MB$
  must be equal to $\MB$. Since there are only a finite number of
  possible zero patterns, the number of fixed points is finite.
  
  Since every limit point is a fixed point by
  \Thm{MM_limit_points}\Part{limit_points}, 
  there are only finitely many limit points. 
  Let $\{\mathcal{N}_p\}$ denote a collection of arbitrarily
  small neighborhoods around each fixed point indexed by $p$. Only finitely many
  iterates $\MB\It{k}$ are in %
  $\mathcal{L}_f(\MB\It{0}) - \cup_p \mathcal{N}_p$.
  So, all but finitely many iterates $\MB\It{k}$
  will be in $ \cup_p \mathcal{N}_p$.  But $\lVert \MB\It{k+1} -
  \MB\It{k} \rVert$ eventually becomes smaller than smallest
  distance between any two neighborhoods by 
  \Thm{MM_limit_points}\Part{successive_iterates}. 
  Therefore the sequence $\MB\It{k}$ must
  belong to one of the neighborhoods for all but finitely many
  $k$. So, any sequence of iterates must eventually converge to
  exactly one of the fixed points of $\psi$.
\end{proof}

We now argue that it is impossible for the MM iterate sequence to
converge to a non-KKT point if it has been appropriately initialized.

\begin{lemma}\label{lem:no_convergence_to_non_KKT}
  Let $f$ be as defined in \Eqn{subproblem-simple} and suppose
  \As{full_row_rank_mm} is satisfied.  Suppose $\MB\It{k} \rightarrow
  \MB\It{*}$ is a convergent sequence of iterates defined by
  \Eqn{mm-iterate} with $\MB\It{0} \geq 0$ and $f(\MB\It{0})$ finite.
  If $\left( \MB\It{0} \right)_{ir} > 0$ for all $(i,r)$ such that
  $(\MPhi(\MB\It{*}))_{ir} > 1$, then $\nabla f (\MB\It{*}) \geq
  \M{0}$.
\end{lemma}

\begin{proof}
  We give a proof by contradiction.
  Suppose there exists $(i,r)$ such that $\left( \MB\It{0}
  \right)_{ir} > 0$ but $(\nabla f (\MB\It{*}))_{ir} < 0$.
  Since $\MB\It{*}$ is a fixed point of $\psi$, we must have
  $[ 1 - (\MPhi(\MB\It{*}))_{ir} ] (\MB\It{*})_{ir} = 0$.
  By our assumption, however $(\nabla f (\MB\It{*}))_{ir} =
  [ 1 - (\MPhi(\MB\It{*}))_{ir} ] < 0$.
  Thus, we must have $(\MB\It{*})_{ir} = 0$.
  On the other hand, %
  $(\MB\It{k})_{ir} > 0$ for all $k$ (proof left to reader).  
  Since $\MPhi(\cdot)$ is a
  continuous function of $\MB$ on $\mathcal{L}_f(\MB\It{0})$, we know
  that there exists some $K$ such that $k > K$ implies $\MB\It{k}$ is
  close enough to $\MB\It{*}$ such that
  $(\nabla f (\MB\It{k}))_{ir} = [ 1 - (\MPhi(\MB\It{k}))_{ir} ] <
  0$. Since $(\MB\It{k})_{ir} > 0$, we have
  $[ 1 - (\MPhi(\MB\It{k}))_{ir} ] (\MB\It{k})_{ir} < 0$,
  which implies
  $(\MB\It{k})_{ir} < (\MB\It{k+1})_{ir}$ for all $k > K$.
  But this contradicts
  $\lim_{k \rightarrow \infty} (\MB\It{k})_{ir} = (\MB\It{*})_{ir} = 0$.
  Hence, the claim.
\end{proof}

We now prove \Thm{subproblem-convergence}.

\begin{proof} [of \Thm{subproblem-convergence}]
  By \Lem{MM_iterates_converge}, the sequence  $\{\MB\It{k}\}$ converges;
  we call the limit point $\MB\It{*}$.
  At this limit point, we have:
  \begin{inparaenum}[(a)]
    \item $\MB\It{*} \geq 0$,
    \item $\nabla  f(\MB\It{*}) \geq 0$ by \Lem{no_convergence_to_non_KKT},
    \item and $\MB\It{*} \Hada \nabla  f(\MB\It{*}) = 0$ by virtue of
      $\MB\It{*}$  being a fixed point of $\psi$.
  \end{inparaenum}
  Thus, the point $\MB\It{*}$ satisfies the KKT conditions  with respect to
  \Eqn{subproblem-simple}. Furthermore, since $f$ is strictly convex by \Lem{strict-convexity},
  we can conclude that $\MB\It{*}$ is the global minimum of $f$.
\end{proof}

\section{Numerical experiment details for \Sec{compare_objectives}} 
\label{sec:compare_objectives-app}
All implementations are from Version 2.5 of Tensor Toolbox for MATLAB
\cite{TTB_Software}. All methods use a common initial guess for the
solution. 
\begin{itemize}
\item \textbf{Lee-Seung LS}: Implemented in \texttt{cp\_nmu} as
  descibed in \cite{BaKo07}. We use the default
  parameters except that the maximum number of iterations
  (\texttt{maxiters}) is set to 200 and the tolerance on the change
  in the fit (\texttt{tol}) is set to $10^{-8}$.
\item \textbf{CP-ALS}: Implemented in \texttt{cp\_als} as described in
  \cite{BaKo07}.  We use the default parameter settings except that 
  the maximum number iterations (\texttt{maxiters}) is 200 and the
  tolerance on the changes in fit (\texttt{tol}) is $10^{-8}$.
\item \textbf{Lee Seung KL}: Implemented in \texttt{cp\_apr} as
  described in this paper. The parameters are set as follows: $k_{\max} = 200$
  (\texttt{maxiters}), $\ell_{\max}=1$ (\texttt{maxinneriters}), $\tau
  = 10^{-8}$ (\texttt{tol}), $\kappa=0$ (\texttt{kappa}),
  $\epsilon=0$ (\texttt{epsilon}).
\item \textbf{CP-APR}: Implemented in \texttt{cp\_apr} as
  described in this paper. The parameters are set as follows:
  $k_{\max} = 200$
  (\texttt{maxiters}), $\ell_{\max}=10$ (\texttt{maxinneriters}), $\tau
  = 10^{-4}$ (\texttt{tol}), $\kappa=10^{-2}$ (\texttt{kappa}),
  $\kappa_{\text{tol}}=10^{-10}$ (\texttt{kappatol}),
  $\epsilon=0$ (\texttt{epsilon}).
\end{itemize}

We compare the methods in terms of their ``factor match score,''
defined as follows. Let $\T{M} = \KTsmall{\Vl; \Mn{A}{1},\dots,\Mn{A}{N}}$
be the true model and let $\Tbar{M} = \KTsmall{\Vbar{\lambda};
  \Mbarn{A}{1},\dots\,\Mbarn{A}{N}}$ be the computed
solution. 
The score of $\Tbar{M}$ is computed as
\begin{gather*}
  \text{score}(\Tbar{M}) = \frac{1}{R}
  \sum_r \left( 1 - \frac{|\xi_r - \bar \xi_r|}{\max\{\xi_r,\bar\xi_r\}}\right)
  \prod_n
  \frac{\MnCTra{a}{n}{r}\MbarnC{a}{n}{r}}
  {\|\MnC{a}{n}{r}\| \|\MbarnC{a}{n}{r}\|}, \\
  \qtext{where}
  \xi_r = \lambda_r \prod_n \|\MnC{A}{n}{r}\|
  \qtext{and}
  \bar\xi_r = \bar\lambda_r \prod_n \|\MbarnC{A}{n}{r}\|.
\end{gather*}
The FMS is a rather abstract measure, so we also give results for the
number of columns in $\Mn{A}{1}$ that are correctly identified. In
other words, we count the number of times that 
the cosine of the angle between the true solution and
the computed solution is greater than 0.95, mathematically,
${\MnCTra{a}{1}{r}\MbarnC{a}{1}{r}}/{\|\MnC{a}{1}{r}\| \|\MbarnC{a}{1}{r}\|} \geq 0.95$.
We use the first mode, but the results are representative of the other
modes.

\section{Additional Enron results}
\label{sec:enron-app}

\Fig{enron-app} illustrates the four components omitted in
\Fig{enron}.

\begin{figure}[htbp]
  \centering
  \EnronFig{1}{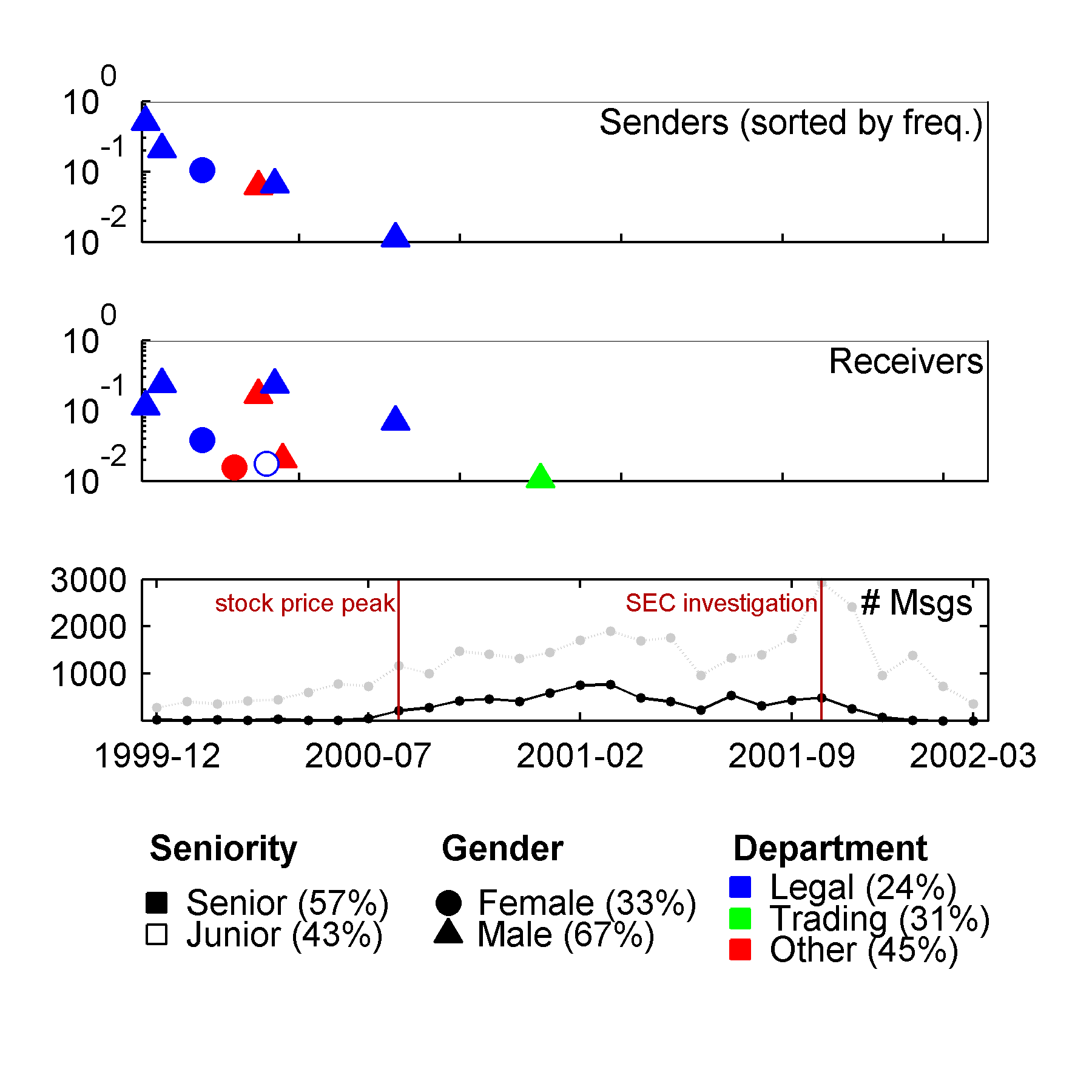}~
  \EnronFig{3}{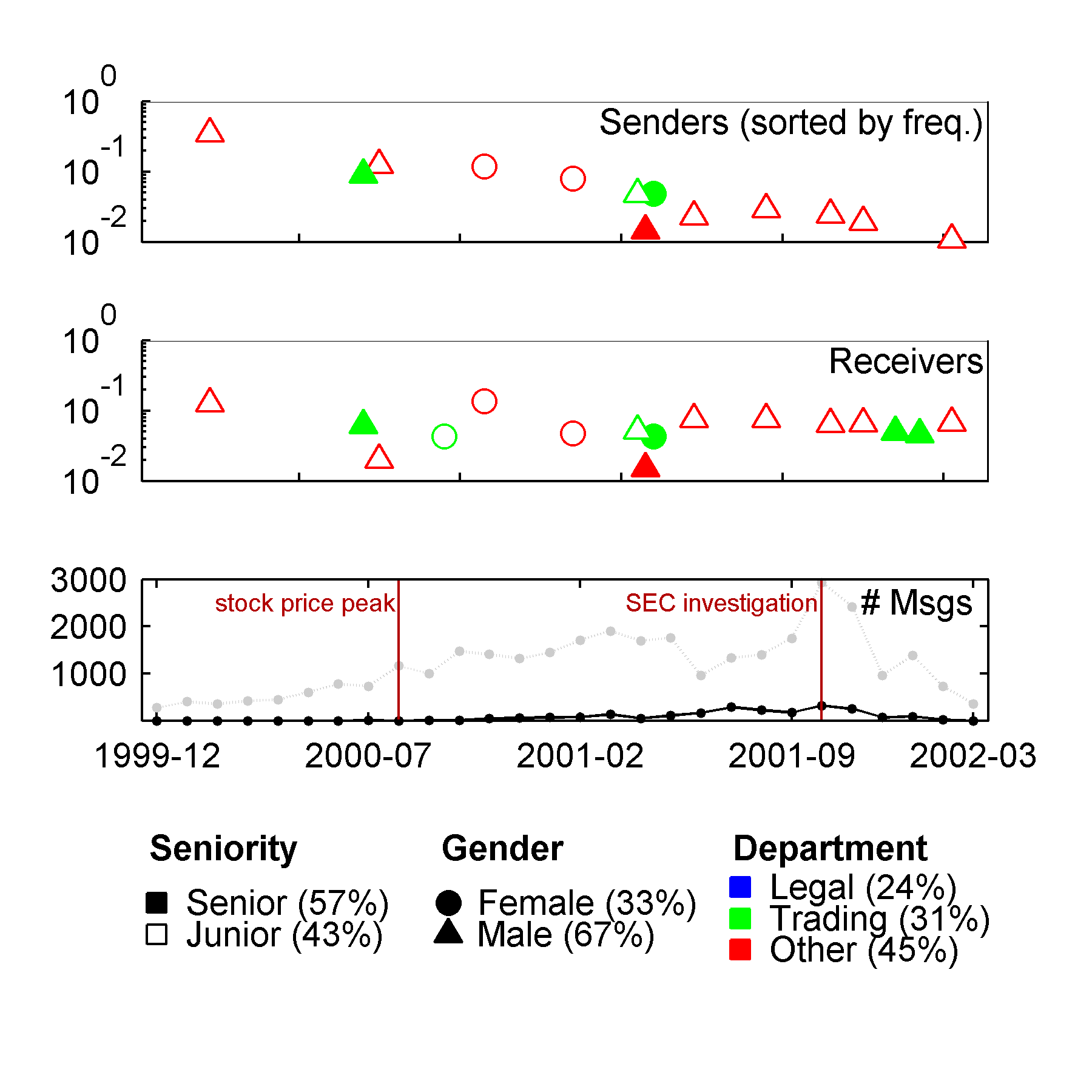}\\
  \EnronFig{4}{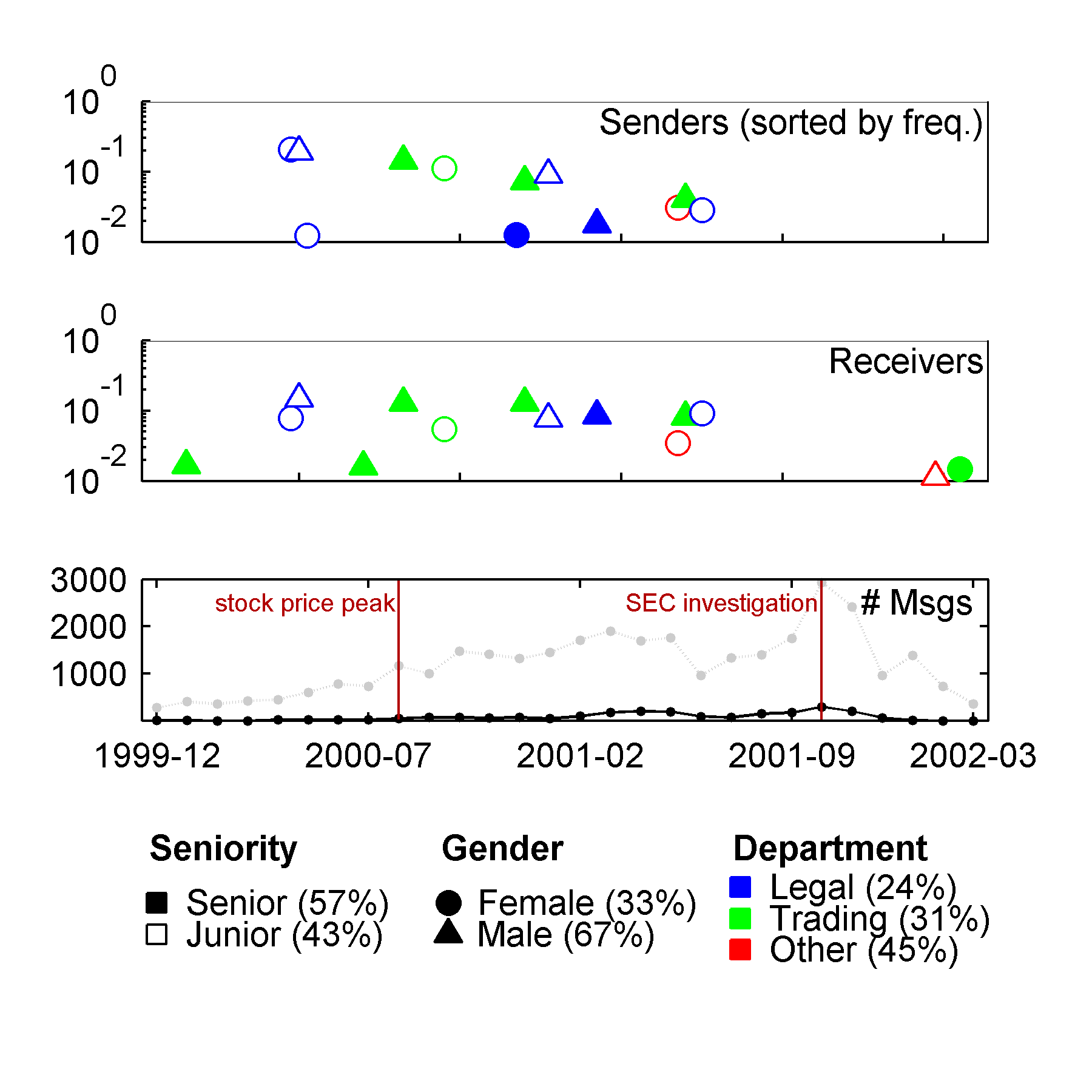}~
  \EnronFig{5}{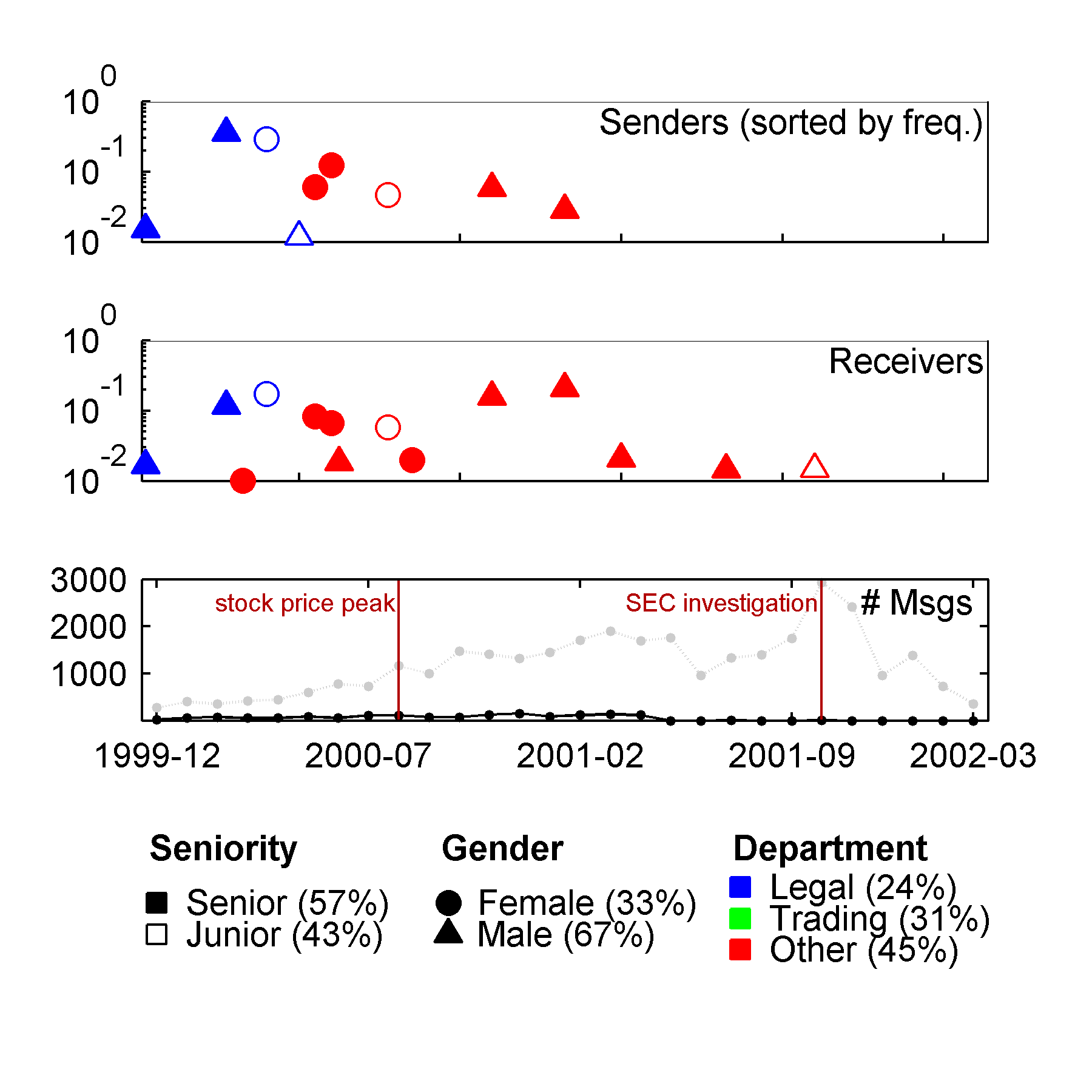}
  \caption{Remaining components from factorizing the Enron data.}
  \label{fig:enron-app}
\end{figure}

\end{document}